\documentclass{amsart}
\usepackage{amsfonts,amssymb,amsmath,amsthm}
\usepackage{url}
\usepackage{enumerate}
\usepackage{bbm}
\usepackage{amssymb}

\newtheorem{thrm}{Theorem}[section]
\newtheorem{lem}[thrm]{Lemma}
\newtheorem{prop}[thrm]{Proposition}
\newtheorem{cor}[thrm]{Corollary}
\theoremstyle{definition}

\newtheorem{remark}[thrm]{Remark}
\numberwithin{equation}{section}

\newcommand{\N}{\ensuremath{{\mathbb N}}}

\newcommand{\R}{\ensuremath{{\mathbb R}}}
\newcommand{\E}{\ensuremath{{\mathbb E}}}
\newcommand{\norm}[1]{\left \lVert#1 \right\rVert}
\newcommand{\abs}[1]{\left\lvert#1 \right\rvert}
\newcommand{\skp}[1]{\left<#1\right>}

\author{David Alonso-Guti\'errez  \and Joscha Prochno}
\address[David Alonso-Guti\'errez]{Department of Mathematical and Statistical Sciences\\
University of Alberta\\
505 Central Academic Building\\
Edmonton T6G 2G1\\
Canada}
\email{alonsogu@ualberta.ca}
\address[Joscha Prochno]{Department of Mathematical and Statistical Sciences\\
University of Alberta\\
605 Central Academic Building\\
Edmonton T6G 2G1\\
Canada}
\email{prochno@ualberta.ca}
\thanks{~}

\begin{document}

\title[Estimating Support Functions via Orlicz Norms]{Estimating Support Functions of Random Polytopes via Orlicz Norms}

\begin{abstract}
We study the expected value of support functions of random polytopes in a certain direction, where the random polytope is given by independent random vectors uniformly distributed in an isotropic convex body. All results are obtained by an utterly novel approach, using probabilistic estimates in connection with Orlicz norms that were not used in this connection before.\end{abstract}
\maketitle

\section{Introduction and Notation} \label{sect1}

The study of random polytopes began with Sylvester and the famous four-point problem nearly 150 years ago. It was the starting point for an extensive study. In their groundbreaking work \cite{key-RSu} from 1963, R\'enyi and Sulanke continued it, studying expectations of various basic functionals of random polytopes. Important quantities are expectations, variances and distributions of those functionals, and their study combines convex geometry, as well as geometric analysis and geometric probability (see also \cite{key-Ba}, \cite{key-Re}). 

In the last 30 years a tremendous effort was made to explore properties of random polytopes as they gained more and more importance due to many important applications and connections to various other fields. Those can be found not only in statistics (extreme points of random samples) and convex geometry (approximation of convex sets), but also in computer science in the analysis of the average complexity of algorithms (\cite{key-PSh}) and optimization (\cite{key-B}), and even in biology (\cite{key-S}). In 1989, Milman and Pajor revealed a deep connection to functional analysis, proving that the expected volume of a certain random simplex is closely related to the isotropic constant of a convex set. In fact, this is a fundamental quantity in convex geometry and the local theory of Banach spaces (\cite{key-MP}). 

Since Gluskin's result \cite{key-G} random polytopes are known to provide
many examples of convex bodies (and related normed spaces) with a
``pathologically  bad'' behavior of various parameters of a linear and
geometric nature (see for instance the survey \cite{key-mt} and references therein).
Consequently, they were also a natural candidate for a potential
counterexample for the hyperplane conjecture. The isotropic constant
of certain classes of random polytopes has been studied in
\cite{key-A}, \cite{key-DGT} and \cite{key-KK}, showing that they do
not provide a counterexample for the hyperplane conjecture.

Some other recent developments in the study of random polytopes can be
found in \cite{key-DGT} or \cite{key-PP}, where the authors studied
the relation between some parameters of a random polytope in an
isotropic convex body and the isotropic constant of the body. Their results provide sharp estimates whenever $n^{1+\delta} \leq N \leq
e^{\sqrt{n}}$ for some $\delta>0$. However, their method does not cover the case where $N\sim
n$ and it seems that a new approach is needed. Therefore, our paper serves this purpose, providing a new tool in the study of random polytopes where results are obtained for the range $n \leq N \leq e^{\sqrt{n}}$. More precisely, we will estimate the expected value of support functions of random polytopes for a fixed direction, using a representation of this parameter via Orlicz norms.

Even though the motivation is of a geometrical nature, the tools we use are mainly probabilistic and analytical, involving Orlicz norms and therefore spaces which naturally appear in Banach space theory. It is interesting that those spaces, as we will see, also naturally appear in the study of certain parameters of random polytopes. Hence, this interplay between convex geometry and classical Orlicz spaces is attractive both from the analytical as well as from the geometrical point of view.\\

Before stating the exact results, and to allow a better understanding, we start with some basic definitions before we go into detail. A convex body $K\subset \R^n$ is a compact convex set with non-empty interior. It is called symmetric if $-x\in K$, whenever $x\in K$. 
We will denote its volume (or Lebesgue measure) by $|\cdot |$. A convex body is said to be in isotropic position if it has volume $1$ and satisfies the following two conditions:
\begin{itemize}
\item$\int_Kxdx=0 \textrm{ (center of mass at 0)}$,
\item$\int_K\langle x,\theta\rangle^2dx=L_K^2\quad \forall \theta\in S^{n-1}$,
\end{itemize}
where $L_K$ is a constant independent of $\theta$, which is called the isotropic constant of $K$. Here, $\skp{\cdot,\cdot}$ denotes the standard scalar product in $\R^n$. 

We will use the notation $a\sim b$ to express that there exist two positive absolute constants $c_1, c_2$ such that $c_1a\leq b\leq c_2 a$ and use $a\sim_{\delta} b$ in case the constants depend on some constant $\delta>0$. Similarly, we write $a \lesssim b$ if there exists a positive absolute constant $c$ such that $a \leq c b$. The letters 
$c, c', C, C', c_1,c_2,\ldots$ will denote positive absolute constants, whose value may change from line to line. We will write $C(r)$ if the constant depends on some parameter $r>0$.

Let $K$ be a convex body and $\theta\in S^{n-1}$ a unit vector. The support function of $K$ in the direction $\theta$ is defined by $h_K(\theta)=\max\{\skp{x,\theta}: x\in K\}$. The mean width of $K$ is 
$$
w(K)=\int_{S^{n-1}}h_K(\theta)d\mu(\theta),
$$
where $d\mu$ denotes the uniform probability measure on $S^{n-1}$.

Given an isotropic convex body $K$, let us consider the random polytope $K_N=\textrm{conv}\{\pm X_1,\dots,\pm X_N\}$, where $X_1,\dots, X_N$ are independent random vectors uniformly distributed in $K$. It is known (see for instance \cite{key-DGT} or \cite{key-Pi}) that the expected value of the mean width of $K_N$ is bounded from above by
 $$
 \E w(K_N)\leq C L_K \sqrt{\log N},
 $$
where $C$ is a positive absolute constant. In \cite{key-DGT} the authors showed that if $N\leq e^{\sqrt n}$
$$
    \E \left( \frac{|K_N|}{|B_2^n|} \right)^{\frac{1}{n}} \geq C L_K \sqrt{\log\frac{N}{n}}.
$$
As a consequence they obtained
$$
    \E w(K_N) \sim_{\delta} L_K \sqrt{\log N},
$$
if the number of random points defining $K_N$ verifies $n^{1+\delta} \leq N \leq e^{\sqrt{n}}$, $\delta>0$ a constant. 

Now, let us be more precise and outline what we will prove and study in the following. First of all, by Fubini's Theorem, the expected value of the mean width of $K_N$ is the average on $S^{n-1}$ of the expected value of the support function of $K_N$ in the direction $\theta$:
\begin{equation}\label{EQU Average of Orlicz}
\E w(K_N)= \E\int_{S^{n-1}} h_{K_N}(\theta)d\mu=\int_{S^{n-1}}\E h_{K_N}(\theta)d\mu.
\end{equation}
Initially, in this paper we are interested in estimating $\E
h_{K_N}(\theta)=\E\max\limits_{1\leq i \leq N}|\langle
X_i,\theta\rangle|$ for a fixed direction $\theta\in S^{n-1}$, but we
will also derive ``high probability'' (in the set of directions) results. In order to do so, we establish a completely new approach applying probabilistic estimates in connection with Orlicz norms. Those were first studied by Kwapie\'n and Sch\"utt in the discrete case in \cite{key-KS1} and \cite{key-KS2}, and later extended by Gordon, Litvak, Sch\"utt and Werner in \cite{key-GLSW} and \cite{key-GLSW2} (for recent developments see also \cite{key-Pr2}, \cite{key-PrR} and \cite{key-PrS}). Using this method to estimate support functions of random polytopes is interesting in itself and introduces a new tool in convex geometry. \\
As we will see, the expected value of the mean width of a random polytope in (\ref{EQU Average of Orlicz}) is equivalent to an average of Orlicz norms, {\it i.e.},
  $$
    \E w(K_N) \sim \int_{S^{n-1}} \norm{(1,...,1)}_{M_{\theta}} d\mu(\theta).
  $$
This, in fact, is not just a nice representation, but a very interesting observation, which bears information concerning the expected value of the mean width, worth to be studied in more detail. Notice that averages of Orlicz norms naturally appear in Functional analysis when studying symmetric subspaces of the classical Banach space $L_1$ (see \cite{key-BDC}, \cite{key-KS1}, \cite{key-Pr} just to mention a few). To be more precise, as shown in \cite{key-KS1} every finite-dimensional symmetric subspace of $L_1$ is $C$-isomorphic to an average of Orlicz spaces (see \cite{key-RS} for the corresponding result for rearrangement invariant spaces). 

In Section \ref{Preliminaries} we will introduce the aforementioned Orlicz norm method that we will use throughout this paper to prove estimates for support functions of random polytopes.

In Section \ref{l_p^n balls}, with this approach, denoting by $e_j$ the canonical basis vectors in $\R^n$, we first compute $\E h_{K_N}(e_j)$ when the isotropic convex body in which $K_N$ lies is the normalized $\ell_p^n$ ball, {\it i.e.}, in $D_p^n=\frac{B_p^n}{|B_p^n|^\frac{1}{n}}$. Namely, using these ideas, we prove the following:

\begin{thrm}\label{mainTheorem}
Let $X_1,\dots, X_N$ be independent random vectors uniformly distributed in $D_p^n$, $1\leq p \leq \infty$, with $n\leq N\leq e^{c'n}$, and $K_N=\textrm{conv}\{\pm X_1,\ldots,\pm X_N\}$. Then, for all $j=1,\ldots,n$,
$$
\E h_{K_N}(e_j)=\E\max_{1\leq i\leq N}|\langle X_i,e_j\rangle|\sim (\log N)^\frac{1}{p}.
$$
\end{thrm}

Many properties of random variables distributed in $\ell_p^n$ balls have already been studied, 
see for instance \cite{key-BGMN}, \cite{key-SZ1} and \cite{key-SZ2}.

By rotational invariance in the Euclidean case, we obtain the same estimate for the expected value of the mean width of a random polytope in $D_2^n$, under milder conditions on the number of points $N$:

\begin{cor}\label{COR Mean width}
Let $X_1,\dots, X_N$ be independent random vectors uniformly distributed in $D_2^n$, with $n\leq N\leq e^n$ and let $K_N=\textrm{conv}\{\pm X_1,\ldots,\pm X_N\}$. Then
$$
\E w(K_N)\sim \sqrt{\log N}.
$$
\end{cor}

In Section \ref{GeneralResult} we will use our approach to give a general upper bound for $\E h_{K_N}(\theta)$ when $K$ is symmetric and under some smoothness conditions on the function 
$h(t)=|K \cap \{ \langle x,\theta\rangle=t \} |^\frac{1}{n-1}$. This general case will include the case when $K=D_p^n$ with $2\leq p < \infty$ and $\theta=e_j$. \\
As proved in \cite{key-PP}, the expected value of the intrinsic volumes (in particular the mean width) of $K_N$ are minimized when $K=D_2^n$. Thus, we have $\E w(K_N)\gtrsim \sqrt{\log N}$ and $\E w(K_N)\sim L_K\sqrt{\log N}$ for those bodies with the isotropic constant bounded.
We prove the existence of directions such that the expected value of the support function in this directions is bounded from above by a constant times $L_K \sqrt{\log N}$ respectively bounded from below by a constant times $L_K \sqrt{\log N}$. In fact, as a consequence we estimate the measure of the set of directions verifying such estimates. It is stated in the following corollary. Notice that the constant $L_K$ appears explicitly also in the lower bound. 
\begin{cor}\label{COR measure estimate}
  Let $n \leq N \leq e^{\sqrt{n}}$, $K$ be an isotropic convex body in $\R^n$ and let $X_1,\ldots, X_N$ be independent random variables uniformly distributed on $K$. 
  Let $K_N=conv\{\pm X_1,\ldots,\pm X_N\}$. For every $r>0$ there exist positive constants $C(r),C_1(r),C_2(r)$ such that     
\begin{eqnarray*}
        \E h_{K_N}(\theta) &\leq&C_1(r) L_K \sqrt{\log
          N}\\
\E h_{K_N}(\theta) &\geq&C_2(r) L_K \sqrt{\log N}
    \end{eqnarray*}
  for a set of directions with measure greater than $
  1-\frac{1}{N^r}$ and $ \frac{C(r)\sqrt{\log N}}{N^r}$ respectively .        
\end{cor}
All the estimates we prove using our approach hold when $n\leq N\leq e^{\sqrt n}$. Thus, our method might provide a tool to prove $\E w(K_N) \sim L_K \sqrt{\log N}$ for this range of $N$ and hence close the gap mentioned in \cite{key-DGT}, where the authors' result was restricted to the case $n^{1+\delta} \leq N \leq e^{\sqrt{n}}$, $\delta>0$, and constants depending on $\delta$. 

\section{Preliminaries}\label {Preliminaries}

A convex function $M:[0,\infty)\to[0,\infty)$ where $M(0)=0$ and $M(t)>0$ for $t>0$ is called an {Orlicz function}. If there is a $t_0>0$ such that for
all $t\leq t_0$ we have $M(t)=0$ then $M$ is called a {degenerated Orlicz function}. The {dual function} $M^*$ of an Orlicz function $M$ is given by the
Legendre transform
  $$
    M^*(x) = \sup_{t\in[0,\infty)}(xt-M(t)).
  $$
Again, $M^*$ is an Orlicz function and $M^{**}=M$. For instance, taking $M(t)=\frac{1}{p}t^p$, $p\geq 1$, the dual function is given by $M^*(t)=\frac{1}{p^*}t^{p^*}$ with $\frac{1}{p^*}+\frac{1}{p}=1$
The $n$-dimensional {Orlicz space} $\ell_M^n$ is $\R^n$ equipped with the norm
  $$
    \norm{x}_M = \inf \left\{ \rho>0 \,:\, \sum_{i=1}^n M\left(\tfrac{\abs{x_i}}{\rho}\right) \leq 1 \right\}.
  $$
In case $M(t)=t^p$, $1\leq p<\infty$, we just have $\norm{\cdot}_M=\norm{\cdot}_p$. For a detailed and thorough introduction to the theory of Orlicz spaces we refer the reader to \cite{key-KR} and \cite{key-RR}.\\

In \cite{key-GLSW2} the authors obtained the following result:

\begin{thrm}[\cite{key-GLSW2} Lemma 5.2]\label{THM Schuett Werner Litvak Gordon}
Let $X_1,\ldots,X_N$ be iid random variables with finite first moments. For all $s\geq 0$ let
$$
  M(s)=\int\limits_0\limits^s \int_{\{\frac{1}{t}\leq |X_1|\}} |X_1| d\mathbb P dt.
$$
Then, for all $x=(x_i)_{i=1}^N\in\R^N$,
$$
\mathbb E \max\limits_{1\leq i \leq N}|x_iX_i|\sim \norm{x}_M.
$$
\end{thrm}

Obviously, the function
\begin{equation}\label{EQU Orlicz function M}
  M(s)=\int_0^s \int_{\{\frac{1}{t}\leq |X_1|\}} |X_1| d\mathbb P dt
\end{equation}
is non-negative and convex, since $\int_{\{\frac{1}{t}\leq \abs{X}\}}\abs{X}d\mathbb P$ is increasing in $t$. Furthermore, we have
$M(0)=0$ and $M$ is continuous. One can easily show, that this Orlicz function $M$ can also be written in
the following way:
  $$
    M(s) = \int_0^s \left(\tfrac{1}{t} \mathbb P (\abs{X}\geq \tfrac{1}{t}) + \int_{\frac{1}{t}}^{\infty} \mathbb P (\abs{X}\geq u) du \right) dt.
  $$
  
As a corollary we obtain the following result, which is the one we use to estimate the support functions of random polytopes.

\begin{cor}\label{CorollaryOrlicz}
Let $X_1,\ldots,X_N$ be iid random vectors in $\R^n$ and let $K_N=\textrm{conv}\{\pm X_1,\ldots,\pm X_N\}$. Let $\theta\in S^{n-1}$ and
$$
  M_{\theta}(s)=\int_0^s \int_{\{\frac{1}{t}\leq |\langle X_1,\theta\rangle|\}} |\langle X_1,\theta\rangle| d\mathbb P dt.
$$
Then
$$
\E h_{K_N}(\theta)\sim\inf \left\{ s>0 \,:\, M_{\theta}\left(\tfrac{1}{s}\right) \leq \tfrac{1}{N} \right\}.
$$
\end{cor}

\section{Random Polytopes in Normalized $\ell_p^n$-Balls}\label{l_p^n balls}

In this section we consider random polytopes $K_N=\textrm{conv}\{\pm X_1,\ldots,\pm X_N\}$, where $X_1,\dots,X_N$ are independent random vectors uniformly distributed in the normalized $\ell_p^n$ ball $D_p^n=\frac{B_p^n}{|B_p^n|^\frac{1}{n}}$. Let us recall that the volume of $B_p^n$ equals
$$
 |B_p^n| = \frac{(\Gamma(1+\tfrac{1}{p}))^n}{\Gamma(1+\tfrac{n}{p})},
$$
 and so, using Stirling's formula, we have that $|B_p^n|^{1/n} \sim \frac{1}{n^\frac{1}{p}}$ and $\frac{|B_p^{n-1}|}{|B_p^n|}\sim n^\frac{1}{p}$.

We are going to estimate $\E h_{K_N}(e_j)$ using the Orlicz norm approach introduced in Section \ref{Preliminaries}. In order to do so, we need to compute the Orlicz function $M$ introduced in Corollary \ref{CorollaryOrlicz}. We are doing this in the following.

\begin{lem}\label{Mlpballs}
Let $1\leq p <\infty$ and $M:[0,\infty)\to[0,\infty)$ be the function
$$
 M(s):= M_{e_j}(s)=\int_0^s \int_{\{x\in D_p^n\,:\,|\langle x,e_j\rangle|\geq\frac{1}{t}\}} |\langle x,e_j\rangle| dx dt.
$$
Then, if $s\leq\frac{1}{|B_p^n|^\frac{1}{n}}$,
\begin{eqnarray}\label{firstexpressionforM}
M\left(\frac{1}{s}\right)&=&\frac{4}{p(n-1+p)}\frac{|B_p^{n-1}|}{|B_p^n|}\int_0^{\cos^{-1}(s|B_p^n|^{\frac{1}{n}})^\frac{p}{2}}\frac{(\sin\theta)^{2\frac{n-1}{p}+3}}{(\cos\theta)^{3-\frac{2}{p}}}d\theta\cr
    &+&\frac{4(2-p)}{p(n-1+p)}\frac{|B_p^{n-1}|}{|B_p^n||B_p^n|^\frac{p-2}{n}} \times \cr
    & \times &
    \int_0^{\cos^{-1}(s|B_p^n|^{\frac{1}{n}})^\frac{p}{2}}\frac{\sin\theta}{(\cos\theta)^{1+\frac{2}{p}}}\int_{(\cos\theta)^\frac{2}{p}|B_p^n|^{-\frac{1}{n}}}^{|B_p^n|^{-\frac{1}{n}}}r^{1-p}(1-|B_p^n|^\frac{p}{n}r^p)^{\frac{n-1}{p}+1}drd\theta.\cr
& &
\end{eqnarray}
Also, if $s\leq\frac{1}{|B_p^n|^\frac{1}{n}}$,
\begin{eqnarray}\label{secondexpressionforM}
M\left(\frac{1}{s}\right)&=&\frac{2}{(n-1+p)\left(n-1+2p\right)}\frac{|B_p^{n-1}|}{|B_p^n|}\frac{(1-s^p|B_p^n|^\frac{p}{n})^{\frac{n-1}{p}+2}}{(s^p|B_p^n|^\frac{p}{n})^{2-\frac{1}{p}}}\cr
&-&\frac{12(p-1)}{p(n-1+p)(n-1+2p)}\frac{|B_p^{n-1}|}{|B_p^n|}\int_0^{\cos^{-1}(s|B_p^n|^{\frac{1}{n}})^\frac{p}{2}}\frac{(\sin\theta)^{2\frac{n-1}{p}+5}}{(\cos\theta)^{5-\frac{2}{p}}}d\theta\cr
&-&\frac{8(2-p)(p-1)}{p\left(n-1+2p\right)(n-1+p)}\frac{|B_p^{n-1}|}{|B_p^n||B_p^n|^\frac{2p-2}{n}}\times\cr
&\times&\int_0^{\cos^{-1}(s|B_p^n|^{\frac{1}{n}})^\frac{p}{2}}\frac{\sin\theta}{(\cos\theta)^{1+\frac{2}{p}}}\int_{(\cos\theta)^\frac{2}{p}|B_p^n|^{-\frac{1}{n}}}^{|B_p^n|^{-\frac{1}{n}}}r^{1-2p}(1-|B_p^n|^\frac{p}{n}r^p)^{\frac{n-1}{p}+2}drd\theta.\cr
& &
\end{eqnarray}
\end{lem}

\begin{proof}
The $(n-1)$-dimensional volume $|D_p^n\cap\{ \langle x,e_j\rangle=t\}|$ equals
$$
\tfrac{|B_p^{n-1}||B_p^n|^{1/n}}{|B_p^n|}(1-|B_p^n|^{p/n}t^p)^{\frac{n-1}{p}} \mathbbm 1_{[-|B_p^n|^{-1/n},|B_p^n|^{-1/n}]}(t).
$$
By Fubini's Theorem we have that if $s\geq |B_p^n|^{1/n}$
  \begin{eqnarray*}
    M(s) & = & 2\frac{|B_p^{n-1}||B_p^n|^\frac{1}{n}}{|B_p^n|} \int_{|B_p^n|^{1/n}}^{s} \int_{\frac{1}{t}}^{|B_p^n|^{-1/n}} r \left( 1-|B_p^n|^{p/n}r^p \right)^{\frac{n-1}{p}} dr dt \\
    & = & 2\frac{|B_p^{n-1}||B_p^n|^\frac{1}{n}}{|B_p^n|}\int_{|B_p^n|^{1/n}}^s\int_{\frac{1}{t}}^{|B_p^n|^{-1/n}}\frac{r^{p-1}}{r^{p-2}}\left( 1-|B_p^n|^{p/n}r^p \right)^{\frac{n-1}{p}} dr dt.\\
  \end{eqnarray*}
Otherwise $M$ is $0$. Integration by parts yields
  \begin{eqnarray*}
    M(s) & = & 2\frac{|B_p^{n-1}||B_p^n|^\frac{1}{n}}{|B_p^n||B_p^n|^\frac{p}{n}} \int_{|B_p^n|^{\frac{1}{n}}}^{s} \left[
         \tfrac{\left(\tfrac{1}{t}\right)^{2-p}\left(1-\tfrac{|B_p^n|^{\frac{p}{n}}}{t^p} \right)^{\frac{n-1}{p}+1}}{n-1+p} + \right. \\
         && \left. + \int_{\frac{1}{t}}^{|B_p^n|^{-\frac{1}{n}}}(2-p)r^{1-p}
         \tfrac{\left(1-|B_p^n|^{\frac{p}{n}}r^p \right)^{\frac{n-1}{p}+1}}{n-1+p}dr\right]dt. \\
  \end{eqnarray*}
  Now, making the change of variables
  $$
    \frac{|B_p^n|^{\frac{1}{n}}}{t} = (\cos\theta)^{\frac{2}{p}} \Longrightarrow \frac{dt}{d\theta}
    =|B_p^n|^{\frac{1}{n}}\frac{2}{p}\frac{\sin\theta}{(\cos\theta)^{1+\frac{2}{p}}},
  $$
  we obtain
    \begin{eqnarray*}
    M(s)&=&\frac{4}{p(n-1+p)}\frac{|B_p^{n-1}|}{|B_p^n|}\int_0^{\cos^{-1}(s^{-1}|B_p^n|^{\frac{1}{n}})^\frac{p}{2}}\frac{(\sin\theta)^{2\frac{n-1}{p}+3}}{(\cos\theta)^{3-\frac{2}{p}}}d\theta\\
    &+&\frac{4(2-p)}{p(n-1+p)}\frac{|B_p^{n-1}|}{|B_p^n||B_p^n|^\frac{p-2}{n}}\times \\
    & \times & \int_0^{\cos^{-1}(s^{-1}|B_p^n|^{\frac{1}{n}})^\frac{p}{2}}\frac{\sin\theta}{(\cos\theta)^{1+\frac{2}{p}}}\int_{(\cos\theta)^\frac{2}{p}|B_p^n|^{-\frac{1}{n}}}^{|B_p^n|^{-\frac{1}{n}}}r^{1-p}(1-|B_p^n|^\frac{p}{n}r^p)^{\frac{n-1}{p}+1}drd\theta.
  \end{eqnarray*}
Therefore
  \begin{eqnarray*}
    M\left(\frac{1}{s}\right)  &=& \frac{4}{p(n-1+p)}\frac{|B_p^{n-1}|}{|B_p^n|}\int_0^{\cos^{-1}(s|B_p^n|^{\frac{1}{n}})^\frac{p}{2}}\frac{(\sin\theta)^{2\frac{n-1}{p}+3}}{(\cos\theta)^{3-\frac{2}{p}}}d\theta\\
    &+&\frac{4(2-p)}{p(n-1+p)}\frac{|B_p^{n-1}|}{|B_p^n||B_p^n|^\frac{p-2}{n}} \times \\
    & \times & \int_0^{\cos^{-1}(s|B_p^n|^{\frac{1}{n}})^\frac{p}{2}}\frac{\sin\theta}{(\cos\theta)^{1+\frac{2}{p}}}\int_{(\cos\theta)^\frac{2}{p}|B_p^n|^{-\frac{1}{n}}}^{|B_p^n|^{-\frac{1}{n}}}r^{1-p}(1-|B_p^n|^\frac{p}{n}r^p)^{\frac{n-1}{p}+1}drd\theta
  \end{eqnarray*}
if $s\leq \frac{1}{|B_p^n|^{\frac{1}{n}}}$ and 0 otherwise, which is the expression in (\ref{firstexpressionforM}).
The first term in the previous sum equals
$$
\frac{4}{p(n-1+p)}\frac{|B_p^{n-1}|}{|B_p^n|}\int_0^{\cos^{-1}(s|B_p^n|^{\frac{1}{n}})^\frac{p}{2}}\frac{(\sin\theta)^{2\frac{n-1}{p}+3}(\cos\theta)}{(\cos\theta)^{4-\frac{2}{p}}}d\theta
$$
and integration by parts yields that this equals
\begin{eqnarray*} \tfrac{2}{p(n-1+p)\left(\frac{n-1}{p}+2\right)}\frac{|B_p^{n-1}|}{|B_p^n|}\left(\tfrac{(1-s^p|B_p^n|^\frac{p}{n})^{\frac{n-1}{p}+2}}{(s^p|B_p^n|^\frac{p}{n})^{2-\frac{1}{p}}}
-\left(4-\frac{2}{p}\right)\int_0^{\cos^{-1}(s|B_p^n|^{\frac{1}{n}})^\frac{p}{2}}\frac{(\sin\theta)^{2\frac{n-1}{p}+5}}{(\cos\theta)^{5-\frac{2}{p}}}d\theta\right).
\end{eqnarray*}
The integral inside the second term equals
\begin{eqnarray*}
\int_{(\cos\theta)^\frac{2}{p}|B_p^n|^{-\frac{1}{n}}}^{|B_p^n|^{-\frac{1}{n}}}r^{1-p}(1-|B_p^n|^\frac{p}{n}r^p)^{\frac{n-1}{p}+1}dr=\int_{(\cos\theta)^\frac{2}{p}|B_p^n|^{-\frac{1}{n}}}^{|B_p^n|^{-\frac{1}{n}}}r^{2-2p}r^{p-1}(1-|B_p^n|^\frac{p}{n}r^p)^{\frac{n-1}{p}+1}dr
\end{eqnarray*}
and, integrating by parts, this equals
\begin{eqnarray*}
\tfrac{1}{p\left(\frac{n-1}{p}+2\right)|B_p^n|^\frac{p}{n}}\left(\tfrac{(\sin\theta)^{2\frac{n-1}{p}+4}|B_p^n|^\frac{2p-2}{n}}{(\cos\theta)^{4-\frac{4}{p}}}-2(p-1)\int_{(\cos\theta)^\frac{2}{p}|B_p^n|^{-\frac{1}{n}}}^{|B_p^n|^{-\frac{1}{n}}}r^{1-2p}(1-|B_p^n|^\frac{p}{n}r^p)^{\frac{n-1}{p}+2}dr\right)
\end{eqnarray*}
and so, the second term above equals
\begin{eqnarray*}
&&
\frac{4(2-p)}{p\left(n-1+2p\right)(n-1+p)}\frac{|B_p^{n-1}|}{|B_p^n|}\int_0^{\cos^{-1}(s|B_p^n|^{\frac{1}{n}})^\frac{p}{2}}\frac{(\sin\theta)^{2\frac{n-1}{p}+5}}{(\cos\theta)^{5-\frac{2}{p}}}d\theta\cr
&-&\frac{8(2-p)(p-1)}{p\left(n-1+2p\right)(n-1+p)}\frac{|B_p^{n-1}|}{|B_p^n||B_p^n|^\frac{2p-2}{n}}\times\cr
&\times&\int_0^{\cos^{-1}(s|B_p^n|^{\frac{1}{n}})^\frac{p}{2}}\frac{\sin\theta}{(\cos\theta)^{1+\frac{2}{p}}}\int_{(\cos\theta)^\frac{2}{p}|B_p^n|^{-\frac{1}{n}}}^{|B_p^n|^{-\frac{1}{n}}}r^{1-2p}(1-|B_p^n|^\frac{p}{n}r^p)^{\frac{n-1}{p}+2}drd\theta.
\end{eqnarray*}
Thus, adding the two terms we have that if $s\leq\frac{1}{|B_p^n|^\frac{1}{n}}$
\begin{eqnarray*}
M\left(\frac{1}{s}\right)&=&\frac{2}{(n-1+p)\left(n-1+2p\right)}\frac{|B_p^{n-1}|}{|B_p^n|}\frac{(1-s^p|B_p^n|^\frac{p}{n})^{\frac{n-1}{p}+2}}{(s^p|B_p^n|^\frac{p}{n})^{2-\frac{1}{p}}}\cr
&-&\frac{12(p-1)}{p(n-1+p)(n-1+2p)}\frac{|B_p^{n-1}|}{|B_p^n|}\int_0^{\cos^{-1}(s|B_p^n|^{\frac{1}{n}})^\frac{p}{2}}\frac{(\sin\theta)^{2\frac{n-1}{p}+5}}{(\cos\theta)^{5-\frac{2}{p}}}d\theta\cr
&-&\frac{8(2-p)(p-1)}{p\left(n-1+2p\right)(n-1+p)}\frac{|B_p^{n-1}|}{|B_p^n||B_p^n|^\frac{2p-2}{n}}\times\cr
&\times&\int_0^{\cos^{-1}(s|B_p^n|^{\frac{1}{n}})^\frac{p}{2}}\frac{\sin\theta}{(\cos\theta)^{1+\frac{2}{p}}}\int_{(\cos\theta)^\frac{2}{p}|B_p^n|^{-\frac{1}{n}}}^{|B_p^n|^{-\frac{1}{n}}}r^{1-2p}(1-|B_p^n|^\frac{p}{n}r^p)^{\frac{n-1}{p}+2}drd\theta,
\end{eqnarray*}
which is the expression in (\ref{secondexpressionforM})
\end{proof}

Now we are going to prove Theorem \ref{mainTheorem}. It will be a consequence of the next two propositions, where we will prove the upper and lower bound for $\E h_{K_N}(e_j)$ respectively.

\begin{prop}
  For every $n,N\in\N$, with $n\leq N$, and every $1\leq p < \infty$ we have that if
  $X_1,\ldots, X_N$  are random vectors uniformly distributed in $D_p^n$ then
    $$
      \mathbb E \max_{1\leq i \leq N} \abs{\skp{X_i,e_j}} \lesssim \left(\log N\right)^\frac{1}{p}.
      $$
\end{prop}

\begin{proof}
  If $p\geq2$ and $s\leq \frac{1}{|B_p^n|^{\frac{1}{n}}}$, the second term in the expression of $M\left(\frac{1}{s}\right)$ given by (\ref{firstexpressionforM}) is negative and so
$$
    M\left(\frac{1}{s}\right) \leq  \frac{4}{p(n-1+p)} \frac{|B_p^{n-1}|}{|B_p^n|}
    \int_0^{\cos^{-1}(s|B_p^n|^{\frac{1}{n}})^\frac{p}{2}} \frac{(\sin\theta)^{\frac{2(n-1)}{p}+3}}{(\cos\theta)^{3-\frac{2}{p}}} d\theta.
$$
  Integration by parts gives
  \begin{eqnarray*}
    M\left(\frac{1}{s}\right) & \leq & \frac{4|B_p^{n-1}|}{p(n-1+p)|B_p^n|} \left[
    \frac{(\sin\theta)^{\frac{2(n-1)}{p}+2}}{(2-\frac{2}{p})(\cos\theta)^{2-\frac{2}{p}}}\big|_0^{\cos^{-1}(s|B_p^n|^{\frac{1}{n}})^\frac{p}{2}} \right. \\
    && - \left. \frac{\frac{2(n-1)}{p}+2}{2-\frac{2}{p}} \int_0^{\cos^{-1}(s|B_p^n|^{\frac{1}{n}})^\frac{p}{2}}
    \frac{(\sin\theta)^{\frac{2(n-1)+1}{p}+1}}{(\cos\theta)^{1-\frac{2}{p}}} d\theta \right] \\
    & \leq & \frac{2|B_p^{n-1}|}{(p-1)(n-1+p)|B_p^n|} \frac{(1-s^p|B_p^n|^{\frac{p}{n}})^{\frac{n-1}{p}+1}}{(s|B_p^n|^{\frac{1}{n}})^{p-1}} \\
    & = & \frac{2|B_p^{n-1}|}{(p-1)(n-1+p)|B_p^n|}\frac{1}{(s|B_p^n|^{\frac{1}{n}})^{p-1}} e^{\frac{n-1+p}{p}\log(1-s^p|B_p^n|^{\frac{p}{n}})}.
  \end{eqnarray*}
  Taking $s_0=\frac{1}{2^\frac{1}{p}|B_p^n|^\frac{1}{n}}\min\left\{\alpha\left(\frac{p}{n-1+p}\right)(\log N),1\right\}^\frac{1}{p}$, $\alpha>0$ to be specified later. Since $s_0^p|B_p^n|^\frac{p}{n}\leq\frac{1}{2}$, there exists a constant $c$ such that
 $$
 M\left(\frac{1}{s_0}\right)\leq\frac{2|B_p^{n-1}|}{(p-1)(n-1+p)|B_p^n|}\frac{1}{(s_0|B_p^n|^{\frac{1}{n}})^{p-1}} e^{-cs_0^p|B_p^n|^\frac{p}{n}\frac{n-1+p}{p}}.
$$
Take $\alpha=\frac{2}{c}$. If the minimum in the definition of $s_0$ is $\frac{1}{2}$ then trivially we have
  $$
  \mathbb E \max_{1\leq i \leq N} \abs{\skp{X_i,e_1}} \leq\frac{1}{|B_p^n|^\frac{1}{n}}.
  $$
\bigskip
If not, then
$$
M\left(\frac{1}{s_0}\right)\leq\frac{2|B_p^{n-1}|e^{-\log N}}{(p-1)(n-1+p)|B_p^n|\left(\frac{1}{c}\frac{p}{n-1+p}\log N\right)^\frac{p-1}{p}}
$$
Since $|B_p^{n-1}|/|B_p^{n}| \sim n^{1/p}$, we get
  \begin{eqnarray*}
    M\left(\frac{1}{s_0}\right) & \leq & \tfrac{Cn^{1/p}}{p^2(n-1+p)^\frac{1}{p}} \frac{1}{  \left(
    \log N\right)^{1-\frac{1}{p}} N } \\
    & =& \tfrac{C}{p^2(1+\frac{p-1}{n})^\frac{1}{p}} \frac{1}{  \left(
    \log N\right)^{1-\frac{1}{p}} N }\leq \frac{1}{N}
  \end{eqnarray*}
  when $N\geq N_0$ for some sufficiently large $N_0\in\N$. Altogether, for $p\geq 2$, we obtain
  $$
    \mathbb E h_{K_N}(e_1) = \mathbb E \max_{1\leq i \leq N} \abs{\skp{X_i,e_1}}
    \leq \frac{C}{|B_p^n|^\frac{1}{n}}\min\left\{\left(\frac{p}{n-1+p}\right)(\log  N),1\right\}^\frac{1}{p},
  $$
  where $C$ is an absolute positive constant.
This minimum is 1 if and only if $\log N\geq 1+\frac{n-1}{p}$. In this case the upper bound we obtain is $\frac{C}{|B_p^n|^\frac{1}{n}}\sim Cn^\frac{1}{p}$. Since $n-1\leq p\log N$ we have that the upper bound $Cn^\frac{1}{p}\leq C(\log N)^\frac{1}{p}$.
If the minimum is not 1, since $|B_p^n|^\frac{1}{n}\sim \frac{1}{n^{\frac{1}{p}}}$, we also obtain an upper bound of the order $(\log N)^\frac{1}{p}$.\\
If $p\in[1,2]$ we use that in the representation of $M\left(\frac{1}{s}\right)$ given by (\ref{secondexpressionforM}) only the first term is positive and so
\begin{eqnarray*}
M\left(\frac{1}{s}\right)&\leq&\frac{2}{(n-1+p)\left(n-1+2p\right)}\frac{|B_p^{n-1}|}{|B_p^n|}\frac{(1-s^p|B_p^n|^\frac{p}{n})^{\frac{n-1}{p}+2}}{(s^p|B_p^n|^\frac{p}{n})^{2-\frac{1}{p}}}\cr
&=&\frac{2}{(n-1+p)\left(n-1+2p\right)}\frac{|B_p^{n-1}|}{|B_p^n|}\frac{e^{\frac{n-1+2p}{p}\log\left(1-s^p|B_p^n|^\frac{p}{n}\right)}}{(s|B_p^n|^\frac{1}{n})^{2p-1}}.\cr
\end{eqnarray*}
Taking $s_0=\frac{1}{2^\frac{1}{p}|B_p^n|^\frac{1}{n}}\min\left\{\alpha\left(\frac{p}{n-1+2p}\right)(\log N),1\right\}^\frac{1}{p}$, $\alpha>0$ to be specified later. Since $s_0^p|B_p^n|^\frac{p}{n}\leq\frac{1}{2}$, there exists a constant such that
$$
 M\left(\frac{1}{s_0}\right)\leq\frac{2}{(n-1+p)(n-1+2p)}\frac{|B_p^{n-1}|}{|B_p^n|}\frac{1}{(s_0|B_p^n|^{\frac{1}{n}})^{2p-1}} e^{-cs_0^p|B_p^n|^\frac{p}{n}\frac{n-1+2p}{p}}.
$$
Take $\alpha=\frac{2}{c}$. If the minimum in the definition of $s_0$ is $\frac{1}{2}$ then trivially we have
  $$
  \mathbb E \max_{1\leq i \leq N} \abs{\skp{X_i,e_1}} \leq\frac{1}{|B_p^n|^\frac{1}{n}}.
  $$
\bigskip
If not, then
$$
M\left(\frac{1}{s_0}\right)\leq\frac{2e^{-\log N}}{(n-1+p)(n-1+2p)}\frac{|B_p^{n-1}|}{|B_p^n|}\frac{1}{\left(\frac{1}{c}\frac{p}{n-1+2p}\log N\right)^\frac{2p-1}{p}}.
$$
Since $|B_p^{n-1}|/|B_p^{n}| \sim n^{1/p}$ and $p\in [1,2]$, we get
  \begin{eqnarray*}
    M(\tfrac{1}{s_0}) & \leq & \tfrac{C}{(\log N)^{2-\frac{1}{p}}N}\leq\frac{1}{N}
  \end{eqnarray*}
  when $N\geq N_0$ for some sufficiently large $N_0\in\N$.  Altogether, for $1\leq p\leq 2$, we obtain
  $$
    \mathbb E h_{K_N}(e_1) = \mathbb E \max_{1\leq i \leq N} \abs{\skp{X_i,e_1}}
    \leq \frac{C}{|B_p^n|^\frac{1}{n}}\min\left\{\frac{\log  N}{n},1\right\}^\frac{1}{p}\leq C(\log N)^\frac{1}{p},
  $$
  where $C$ is an absolute positive constant.
\end{proof}

In order to prove the lower bound for $\E h_{K_N}(e_j)$ we need the two following technical results:

\begin{lem}\label{LEM General Recursion for Sin}
  Let $\alpha,\beta \in\R\setminus\{-1\}$. Then we have
    $$
      \int \sin^{\alpha}(\theta) \cos^{\beta}(\theta) d\theta = \frac{\sin^{\alpha+1}(\theta) \cos^{\beta+1}(\theta)}{\alpha +1}
      + \tfrac{\alpha + \beta + 2}{\alpha +1} \int \sin^{\alpha+2}(\theta) \cos^{\beta}(\theta) d\theta.
    $$
\end{lem}
\begin{proof}
  We consider $\int \sin^{\alpha+2}(\theta) \cos^{\beta}(\theta) d\theta$. Integration by parts yields
    \begin{eqnarray*}
      \int \sin^{\alpha+2}(\theta) \cos^{\beta}(\theta) d\theta & = & - \frac{\sin^{\alpha+1}(\theta) \cos^{\beta+1}(\theta)}{\beta +1}
      + \tfrac{\alpha+1}{\beta+1} \int \sin^{\alpha}(\theta) \cos^{\beta+2}(\theta) d\theta.
    \end{eqnarray*}
  Since $\cos^{\beta+2}(\theta) = \cos^{\beta}(\theta)(1-\sin^2(\theta))$, we obtain
     \begin{eqnarray*}
      && \int \sin^{\alpha+2}(\theta) \cos^{\beta}(\theta) d\theta \\
      & = & - \frac{\sin^{\alpha+1}(\theta) \cos^{\beta+1}(\theta)}{\beta +1}
      + \tfrac{\alpha+1}{\beta+1} \int \sin^{\alpha}(\theta) \cos^{\beta}(\theta) d\theta
      -\tfrac{\alpha+1}{\beta+1} \int \sin^{\alpha+2}(\theta) \cos^{\beta}(\theta) d\theta.
    \end{eqnarray*}
  Thus
    $$
      \tfrac{\alpha+\beta+2}{\beta+1} \int \sin^{\alpha+2}(\theta) \cos^{\beta}(\theta) d\theta
      = - \frac{\sin^{\alpha+1}(\theta) \cos^{\beta+1}(\theta)}{\beta +1} + \tfrac{\alpha+1}{\beta+1} \int \sin^{\alpha}(\theta) \cos^{\beta}(\theta) d\theta,
    $$
  and so
    $$
      \int \sin^{\alpha}(\theta) \cos^{\beta}(\theta) d\theta = \frac{\sin^{\alpha+1}(\theta) \cos^{\beta+1}(\theta)}{\alpha +1}
      + \tfrac{\alpha + \beta + 2}{\alpha +1} \int \sin^{\alpha+2}(\theta) \cos^{\beta}(\theta) d\theta.
    $$
\end{proof}

As a corollary we obtain the $k$-th iteration of Lemma \ref{LEM General Recursion for Sin}.

\begin{cor}\label{COR General Recursion for Sin k steps}
  Let $\alpha,\beta \in\R\setminus\{-1\}$. Then, for any $k\in\N$, we have
    \begin{eqnarray*}
      \int \sin^{\alpha}(\theta) \cos^{\beta}(\theta) d\theta & = & \frac{\sin^{\alpha+1}(\theta) \cos^{\beta+1}(\theta)}{\alpha +1}
      + \frac{\alpha+\beta+2}{(\alpha +1)(\alpha+3)}\sin^{\alpha+3}(\theta) \cos^{\beta+1}(\theta)+ \\
      &+& \frac{(\alpha+\beta+2)(\alpha+\beta+4)}{(\alpha +1)(\alpha+3)(\alpha+5)}\sin^{\alpha+5}(\theta) \cos^{\beta+1}(\theta)+\ldots+\\
      &+& \frac{(\alpha+\beta+2)\cdots(\alpha+\beta+2k)}{(\alpha +1)\cdots(\alpha+2k+1)}\sin^{\alpha+2k+1}(\theta) \cos^{\beta+1}(\theta)+\\
      &+& \frac{(\alpha+\beta+2)\cdots(\alpha+\beta+2k+2)}{(\alpha +1)\cdots(\alpha+2k+1)} \int \sin^{\alpha+2k+2}(\theta)\cos^{\beta}(\theta) d\theta.
    \end{eqnarray*}
\end{cor}

We will now prove the lower estimate.

\begin{prop}
  There exists a positive absolute constant $c'$, such that for every $n,N\in\N$, with $n \leq N \leq e^{c'n}$, and every $1\leq p < \infty$, we have that if
  $X_1,\ldots, X_N$ are independent random vectors uniformly distributed on $D_p^n$ then
    $$
      \mathbb E \max_{1\leq i \leq N} \abs{\skp{X_i,e_j}} \gtrsim \left(\log N\right)^\frac{1}{p}.
    $$
\end{prop}
\begin{proof}
  We start with the case $1 < p \leq 2$ where we use the recursion formula. Since $1<p\leq 2$ we have, using the representation of $M$ in (\ref{firstexpressionforM}) that
    $$
      M\left(\frac{1}{s}\right) \geq \frac{4}{p(n-1+p)}\frac{|B_p^{n-1}|}{|B_p^n|}\int_0^{\cos^{-1}(s^{-1}|B_p^n|^{\frac{1}{n}})^\frac{p}{2}}
      (\sin\theta)^{2\frac{n-1}{p}+3} (\cos\theta)^{\frac{2}{p}-3}d\theta.
    $$
  Using Corollary \ref{COR General Recursion for Sin k steps} with $\alpha = \frac{2n}{p}-\frac{2}{p}+3$ and $\beta = \frac{2}{p}-3$, we have $-1 \leq \beta+1 <0$, and for any $k\in\N$ we get
    \begin{eqnarray*}
      M\left(\frac{1}{s}\right) & \geq & \frac{4}{p(n-1+p)}\frac{|B_p^{n-1}|}{|B_p^n|} \left[ \frac{(\cos\theta)^{\beta+1}}{\alpha+1}
      \left\{ (\sin\theta)^{\alpha +1} + \frac{\alpha+\beta+2}{\alpha+3}(\sin\theta)^{\alpha+3} + \ldots \right. \right. \\
      &+&  \left. \left. \frac{(\alpha+\beta+2)\cdots(\alpha + \beta +2k)}{(\alpha+3)\cdots(\alpha+2k+1)} (\sin\theta)^{\alpha+2k+1} \right\}
      \Big|_{0}^{\cos^{-1}((s|B_p^{n}|^{\frac{1}{n}})^{\frac{p}{2}})} \right].
    \end{eqnarray*}
  Since $\beta+1=\frac{2}{p}(1-p)$, we get
    \begin{eqnarray*}
      M\left(\frac{1}{s}\right) & \geq & \tfrac{4}{p(n-1+p)}\tfrac{|B_p^{n-1}|}{|B_p^n|(\alpha +1)(s|B_p^n|^{\frac{1}{n}})^{p-1}}
      \left[ (1-s^p|B_p^n|^{\frac{p}{n}})^{\frac{\alpha+1}{2}} + \right. \\
      &+& \left. \tfrac{\alpha+\beta+2}{\alpha+3}(1-s^p|B_p^n|^{\frac{p}{n}})^{\frac{\alpha+3}{2}} + \ldots + \tfrac{(\alpha+\beta +2)\cdots(\alpha+\beta+2k)}{(\alpha+3)\cdots(\alpha+2k+1)}(1-s^p|B_p^n|^{\frac{p}{n}})^{\frac{\alpha+2k+1}{2}}\right] \\
      & \geq & \tfrac{4}{p(n-1+p)}\tfrac{|B_p^{n-1}|(1-s^p|B_p^n|^{\frac{p}{n}})^{\frac{\alpha+2k+1}{2}}}{|B_p^n|(\alpha +1)(s|B_p^n|^{\frac{1}{n}})^{p-1}}
      \left[ 1+(1-\tfrac{1-\beta}{\alpha+3}) + \right. \\
      &+&\left. (1-\tfrac{1-\beta}{\alpha+3})(1-\tfrac{1-\beta}{\alpha+5}) + \ldots +
      (1-\tfrac{1-\beta}{\alpha+3})(1-\tfrac{1-\beta}{\alpha+5})\cdots(1-\tfrac{1-\beta}{\alpha+2k+1}) \right]. \\
      & \geq & \frac{4(k+1)}{p(n-1+p)}\frac{|B_p^{n-1}|(1-s^p|B_p^n|^{\frac{p}{n}})^{\frac{\alpha+2k+1}{2}}}{|B_p^n|(\alpha +1)(s|B_p^n|^{\frac{1}{n}})^{p-1}}
      \left(1-\frac{1-\beta}{\alpha+2k+1}\right)^k.
    \end{eqnarray*}
  So this yields
    $$
      M\left(\frac{1}{s}\right) \geq \frac{2(k+1)}{p(n-1+p)}\frac{|B_p^{n-1}|(1-s^p|B_p^n|^{\frac{p}{n}})^{\frac{n-1}{p}+k+2}}{|B_p^n|
      (\frac{n-1}{p}+2)(s|B_p^n|^{\frac{1}{n}})^{p-1}} \left(1-\frac{2-\frac{1}{p}}{\frac{n-1}{p}+k+1}\right)^k.
    $$
  If we choose $k=n$ and take into account that $1<p\leq 2$, we get
    $$
      M\left(\frac{1}{s}\right) \geq C \frac{|B_p^{n-1}|}{|B_p^{n}|} \frac{e^{\tfrac{n-1+np}{p}\log(1-s^p|B_p^n|^{\frac{p}{n}})}}{(n-1+2p)(s|B_p^n|^{\frac{1}{n}})^{p-1}}.
    $$
  We take $s_0=\frac{\gamma^{\frac{1}{p}}(\log N)^{\frac{1}{p}}}{|B_p^n|^{\frac{1}{n}}n^{\frac{1}{p}}}$with $\gamma$ a constant to be chosen later. Then, since $N\leq e^n$, we obtain
    $$
      M\left(\frac{1}{s_0}\right) \geq C \frac{|B_p^{n-1}|}{|B_p^{n}|} \frac{e^{-c_1\gamma\log N}}{(n-1+2p)(\gamma\frac{\log N}{n})^{1-\frac{1}{p}}}
      \geq \frac{C'}{N^{c_1\gamma}(\gamma \log N)^{1-\frac{1}{p}}}.
    $$
  Choosing $\gamma$ small enough, so that $c_1\gamma < 1$, we get
    $$
      M\left(\frac{1}{s_0}\right) \geq \frac{1}{N}
    $$
  if $N \geq N_0$ for some $N_0\in\N$ large enough. Therefore, there exists an absolute positive constant $c$ such that
    $$
      \mathbb E \max_{1\leq i \leq N} \abs{\skp{X_i,e_j}} \geq c (\log N)^{\frac{1}{p}}.
    $$
  Now, let us consider the easier case where $p=1$. In this case, we have
    $$
      M\left(\frac{1}{s}\right) = \tfrac{2}{n(n+1)}\tfrac{|B_1^{n-1}|}{|B_1^n|}\frac{(1-s|B_1^n|^{\frac{1}{n}})^{n+1}}{s|B_1^n|^{\frac{1}{n}}}.
    $$
  If we now choose $s_0=\alpha \log N$, where $\alpha$ is a constant to be chosen later we obtain
  $$
      M\left(\frac{1}{s_0}\right) \geq \frac{C}{N^{c\alpha}\log N}
    $$
  and so, choosing $\alpha$ a constant small enough so that $c\alpha <1$ we obtain that
 $$
  M\left(\frac{1}{s_0}\right) \geq \frac{1}{N}
  $$
  whenever $N\geq N_0$. Therefore, if $p=1$ there exists an absolute positive constant $c$ such that
    $$
      \mathbb E \max_{1\leq i \leq N} \abs{\skp{X_i,e_j}} \geq c (\log N).
    $$
Now, let's treat the case $2\leq p$. We will assume that $p-1\leq c\frac{n}{\alpha \log N}$, where $\alpha$ is a constant that will be determined later and $c$ is an absolute constant small enough. We will also assume that $p\leq N^\frac{1}{4}$. We have seen that the second term in (\ref{firstexpressionforM}) equals
\begin{eqnarray*}
&&
\frac{4(2-p)}{p\left(n-1+2p\right)(n-1+p)}\frac{|B_p^{n-1}|}{|B_p^n|}\int_0^{\cos^{-1}(s|B_p^n|^{\frac{1}{n}})^\frac{p}{2}}\frac{(\sin\theta)^{2\frac{n-1}{p}+5}}{(\cos\theta)^{5-\frac{2}{p}}}d\theta\cr
&-&\frac{8(2-p)(p-1)}{p\left(n-1+2p\right)(n-1+p)}\frac{|B_p^{n-1}|}{|B_p^n||B_p^n|^\frac{2p-2}{n}}\times\cr
&\times&\int_0^{\cos^{-1}(s|B_p^n|^{\frac{1}{n}})^\frac{p}{2}}\frac{\sin\theta}{(\cos\theta)^{1+\frac{2}{p}}}\int_{(\cos\theta)^\frac{2}{p}|B_p^n|^{-\frac{1}{n}}}^{|B_p^n|^{-\frac{1}{n}}}r^{1-2p}(1-|B_p^n|^\frac{p}{n}r^p)^{\frac{n-1}{p}+2}drd\theta
\end{eqnarray*}
and so if $p\geq 2$ the second term in the expression (\ref{firstexpressionforM}) defining $M\left(\frac{1}{s}\right)$ is greater than or equal to
$$
\frac{4(2-p)}{p\left(n-1+2p\right)(n-1+p)}\frac{|B_p^{n-1}|}{|B_p^n|}\int_0^{\cos^{-1}(s|B_p^n|^{\frac{1}{n}})^\frac{p}{2}}\frac{(\sin\theta)^{2\frac{n-1}{p}+5}}{(\cos\theta)^{5-\frac{2}{p}}}d\theta.
$$
Integration by parts yields that this quantity equals
\begin{eqnarray*}
&&
\frac{4(2-p)}{p(n-1+p)(n-1+2p)}\frac{|B_p^{n-1}|}{|B_p^n|}\times\cr
&\times&\left(\frac{(1-s^p|B_p^n|^\frac{p}{n})^{\frac{n-1}{p}+2}}{\left(4-\frac{2}{p}\right)\left(s|B_p^n|^\frac{1}{n}\right)^{2p-1}}
-\frac{2\frac{n-1}{p}+4}{4-\frac{2}{p}}\int_0^{\cos^{-1}(s|B_p^n|^{\frac{1}{n}})^\frac{p}{2}}\frac{(\sin\theta)^{2\frac{n-1}{p}+3}}{(\cos\theta)^{3-\frac{2}{p}}}d\theta\right).
\end{eqnarray*}
Thus, putting this together with the first term we have that if $p\geq 2$
\begin{eqnarray*}\label{IEQ M first term}
M\left(\frac{1}{s}\right)&\geq&\frac{12(p-1)}{p(2p-1)(n-1+p)}\frac{|B_p^{n-1}|}{|B_p^n|}\int_0^{\cos^{-1}(s|B_p^n|^{\frac{1}{n}})^\frac{p}{2}}\frac{(\sin\theta)^{2\frac{n-1}{p}+3}}{(\cos\theta)^{3-\frac{2}{p}}}d\theta\cr
&-&\frac{2(p-2)}{(n-1+p)(n-1+2p)(2p-1)}\frac{|B_p^{n-1}|}{|B_p^n|}\frac{(1-s^p|B_p^n|^\frac{p}{n})^{\frac{n-1}{p}+2}}{\left(s|B_p^n|^\frac{1}{n}\right)^{2p-1}}.
\end{eqnarray*}
Using integration by parts, the first term in the previous expression equals
\begin{eqnarray*}
&&
\frac{6}{(2p-1)(n-1+p)}\frac{|B_p^{n-1}|}{|B_p^n|}\frac{(1-s^p|B_p^n|^\frac{p}{n})^{\frac{n-1}{p}+1}}{(s|B_p^n|^\frac{1}{n})^{p-1}}\cr
&-&\frac{12}{p(2p-1)}\frac{|B_p^{n-1}|}{|B_p^n|}\int_0^{\cos^{-1}(s|B_p^n|^{\frac{1}{n}})^\frac{p}{2}}\frac{(\sin\theta)^{2\frac{n-1}{p}+1}}{(\cos\theta)^{1-\frac{2}{p}}}d\theta.
\end{eqnarray*}
Using the recursion formula in Corollary \ref{COR General Recursion for Sin k steps} we obtain that for any $k\in\N$ this quantity equals
\begin{eqnarray*}
&&
\tfrac{6}{(2p-1)(n-1+p)}\frac{|B_p^{n-1}|}{|B_p^n|}\frac{(1-s^p|B_p^n|^\frac{p}{n})^{\frac{n-1}{p}+1}}{(s|B_p^n|^\frac{1}{n})^{p-1}}\cr
&-&\tfrac{6}{(2p-1)(n-1+p)}\frac{|B_p^{n-1}|}{|B_p^n|}(s|B_p^n|^\frac{1}{n})(1-s^p|B_p^n|^\frac{p}{n})^{\frac{n-1}{p}+1}\cr
&-&\tfrac{6\left(2\frac{n}{p}+2\right)}{(2p-1)(n-1+p)\left(2\frac{n-1}{p}+4\right)}\frac{|B_p^{n-1}|}{|B_p^n|}(s|B_p^n|^\frac{1}{n})(1-s^p|B_p^n|^\frac{p}{n})^{\frac{n-1}{p}+2}\cr
&-&\tfrac{6\left(2\frac{n}{p}+2\right)\left(2\frac{n}{p}+4\right)}{(2p-1)(n-1+p)\left(2\frac{n-1}{p}+4\right)\left(2\frac{n-1}{p}+6\right)}\frac{|B_p^{n-1}|}{|B_p^n|}(s|B_p^n|^\frac{1}{n})(1-s^p|B_p^n|^\frac{p}{n})^{\frac{n-1}{p}+3}\cr
&-&\dots\cr
&-&\tfrac{6\left(2\frac{n}{p}+2\right)\left(2\frac{n}{p}+4\right)\dots\left(2\frac{n}{p}+2k-2\right)}{(2p-1)(n-1+p)\left(2\frac{n-1}{p}+4\right)\left(2\frac{n-1}{p}+6\right)\dots\left(2\frac{n-1}{p}+2k\right)}\frac{|B_p^{n-1}|}{|B_p^n|}(s|B_p^n|^\frac{1}{n})(1-s^p|B_p^n|^\frac{p}{n})^{\frac{n-1}{p}+k}\cr
&-&\tfrac{6\left(2\frac{n}{p}+2\right)\left(2\frac{n}{p}+4\right)\dots\left(2\frac{n}{p}+2k-2\right)\left(2\frac{n-1}{p}+2k\right)}{(2p-1)(n-1+p)\left(2\frac{n-1}{p}+4\right)\left(2\frac{n-1}{p}+6\right)\dots\left(2\frac{n-1}{p}+2k\right)}\frac{|B_p^{n-1}|}{|B_p^n|}\times\cr
&\times&\int_0^{\cos^{-1}(s|B_p^n|^{\frac{1}{n}})^\frac{p}{2}}\frac{(\sin\theta)^{2\frac{n-1}{p}+2k+1}(\cos\theta)}{(\cos\theta)^{2-\frac{2}{p}}}d\theta.
\end{eqnarray*}
Estimating the cosine in the denominator inside the integral by the value at its extreme point, we obtain that this quantity is greater than
\begin{eqnarray*}
&&
\tfrac{6}{(2p-1)(n-1+p)}\frac{|B_p^{n-1}|}{|B_p^n|}\frac{1}{(s|B_p^n|^{\frac{1}{n}})^{p-1}}\Big((1-s^p|B_p^n|^\frac{p}{n})^{\frac{n-1}{p}+1}-s^p|B_p^n|^\frac{p}{n}(1-s^p|B_p^n|^\frac{p}{n})^{\frac{n-1}{p}+1}\cr
&-&\tfrac{\left(2\frac{n}{p}+2\right)}{\left(2\frac{n-1}{p}+4\right)}s^p|B_p^n|^\frac{p}{n}(1-s^p|B_p^n|^\frac{p}{n})^{\frac{n-1}{p}+2}-\tfrac{\left(2\frac{n}{p}+2\right)\left(2\frac{n}{p}+4\right)}{\left(2\frac{n-1}{p}+4\right)\left(2\frac{n-1}{p}+6\right)}s^p|B_p^n|^\frac{p}{n}(1-s^p|B_p^n|^\frac{p}{n})^{\frac{n-1}{p}+3}\cr
&-&\dots-\tfrac{\left(2\frac{n}{p}+2\right)\left(2\frac{n}{p}+4\right)\dots\left(2\frac{n}{p}+2k-2\right)}{\left(2\frac{n-1}{p}+4\right)\left(2\frac{n-1}{p}+6\right)\dots\left(2\frac{n-1}{p}+2k\right)}s^p|B_p^{n}|^\frac{p}{n}(1-s^p|B_p^n|^\frac{p}{n})^{\frac{n-1}{p}+k}\cr
&-&\tfrac{\left(2\frac{n}{p}+2\right)\left(2\frac{n}{p}+4\right)\dots\left(2\frac{n}{p}+2k-2\right)\left(2\frac{n}{p}+2k\right)}{\left(2\frac{n-1}{p}+4\right)\left(2\frac{n-1}{p}+6\right)\dots\left(2\frac{n-1}{p}+2k+2\right)}(1-s^p|B_p^n|^\frac{p}{n})^{\frac{n-1}{p}+k+1}\Big).\cr
\end{eqnarray*}
Since for every $m$ we have that $\frac{2\frac{n}{p}+2m}{2\frac{n-1}{p}+2m+2}=1-\frac{2-\frac{2}{p}}{2\frac{n-1}{p}+2m+2}\leq 1$, this expression is greater than
\begin{eqnarray*}
&&
\tfrac{6|B_p^{n-1}|(1-s^p|B_p^n|^\frac{p}{n})^{\frac{n-1}{p}+k+1}}{(2p-1)(n-1+p)|B_p^n|(s|B_p^n|^{\frac{1}{n}})^{p-1}}\left(1-\frac{\left(2\frac{n}{p}+2\right)\left(2\frac{n}{p}+4\right)\dots\left(2\frac{n}{p}+2k-2\right)\left(2\frac{n}{p}+2k\right)}{\left(2\frac{n-1}{p}+4\right)\left(2\frac{n-1}{p}+6\right)\dots\left(2\frac{n-1}{p}+2k+2\right)}\right)\cr
&&\geq\tfrac{6}{(2p-1)(n-1+p)}\frac{|B_p^{n-1}|}{|B_p^n|}\frac{(1-s^p|B_p^n|^\frac{p}{n})^{\frac{n-1}{p}+k+1}}{(s|B_p^n|^{\frac{1}{n}})^{p-1}}\left(1-\left(1-\frac{p-1}{n-1+(k+1)p}\right)^k\right)\cr
&&=\tfrac{6}{(2p-1)(n-1+p)}\frac{|B_p^{n-1}|}{|B_p^n|}\frac{(1-s^p|B_p^n|^\frac{p}{n})^{\frac{n-1}{p}+k+1}}{(s|B_p^n|^{\frac{1}{n}})^{p-1}}\left(1-e^{k\log\left(1-\frac{p-1}{n-1+(k+1)p}\right)}\right).
\end{eqnarray*}
Hence,
  \begin{eqnarray*}
    M\left(\frac{1}{s}\right) & \geq & \frac{6}{(2p-1)(n-1+p)}\frac{|B_p^{n-1}|}{|B_p^n|}\frac{(1-s^p|B_p^n|^\frac{p}{n})^{\frac{n-1}{p}+k+1}}{(s|B_p^n|^{\frac{1}{n}})^{p-1}}\left(1-e^{k\log\left(1-\frac{p-1}{n-1+(k+1)p}\right)}\right)\\
    &-& \frac{2(p-2)}{(n-1+p)(n-1+2p)(2p-1)}\frac{|B_p^{n-1}|}{|B_p^n|}\frac{(1-s^p|B_p^n|^\frac{p}{n})^{\frac{n-1}{p}+2}}{\left(s|B_p^n|^\frac{1}{n}\right)^{2p-1}}\\
    &=&\frac{1}{(2p-1)(n-1+p)}\frac{|B_p^{n-1}|}{|B_p^n|}\frac{(1-s^p|B_p^n|^\frac{p}{n})^{\frac{n-1}{p}+2}}{(s|B_p^n|^\frac{1}{n})^{p-1}}\times\\
    &\times&\left(6(1-s^p|B_p^n|^\frac{p}{n})^{k-1}\left(1-e^{k\log\left(1-\frac{p-1}{n-1+(k+1)p}\right)}\right)-\frac{2(p-2)}{(n-1+2p)(s^p|B_p^n|^\frac{p}{n})}\right)
  \end{eqnarray*}
We take
  $$
      s_0 = \frac{ \alpha^{\frac{1}{p}} (p-1)^\frac{1}{p}}{ |B_p^n|^{\frac{1}{n}} n^{\frac{1}{p}} } \left( \log N \right)^{\frac{1}{p}},
  $$
Then,
\begin{eqnarray*}
\frac{2(p-2)}{(n-1+2p)(s_0^p|B_p^n|^\frac{p}{n})}\leq\frac{2}{\left(1-\frac{1}{n}+\frac{2p}{n}\right)\alpha\log N}\leq\frac{2.1}{\alpha\log N}
\end{eqnarray*}
if $n\geq n_0$.
On the other hand, choosing $k$ so that $k+1=\frac{2n}{\alpha(p-1)\log N}$ we have
\begin{eqnarray*}
&&6(1-s_0^p|B_p^n|^\frac{p}{n})^{k-1}\left(1-e^{k\log\left(1-\frac{p-1}{n-1+(k+1)p}\right)}\right)\\
&&\geq 6\left(1-\frac{\alpha (p-1)\log N}{n}\right)^\frac{2 n}{\alpha(p-1)\log N}\left(1-e^{\left(\frac{2 n}{(p-1)\alpha \log N}-1\right)\log\left(1-\frac{p-1}{n-1+\frac{p}{p-1}\frac{2 n}{\alpha \log N}}\right)}\right)\\
&&\geq 6e^{-1}\left(1-e^{\left(\frac{2 n}{(p-1)\alpha \log N}-1\right)\log\left(1-\frac{p-1}{n-1+\frac{p}{p-1}\frac{2 n}{\alpha \log N}}\right)}\right),\\
\end{eqnarray*}
where  the last inequality holds because our assumptions on $p$.
This last quantity is greater than
\begin{eqnarray*}
&&6e^{-1}\left(1-e^{-\left(\frac{2 n}{(p-1)\alpha \log N}-1\right) \frac{p-1}{n-1+\frac{p}{p-1}\frac{2n}{\alpha \log N}}}\right)=6e^{-1}\left(1-e^{-\frac{2}{\alpha \log N}\frac{1-\frac{(p-1)\alpha\log N}{n}}{1-\frac{1}{n}+\frac{p}{p-1}\frac{1}{2\alpha \log N}}}\right)\\
&&\geq 6e^{-1}\left(1-e^{-\frac{2(1-c)}{\alpha\log N}}\right)\geq \frac{6e^{-1}(1-c)}{\alpha \log N}
\end{eqnarray*}
if $N\geq N_0$. Taking $c$ small enough so that $6e^{-1}(1-c)>2.1$, we have that
\begin{eqnarray*}
M\left(\frac{1}{s_0}\right)&\geq&\frac{C}{p^2 N^{c_1\alpha}(\alpha \log N)^{2-\frac{2}{p}}}\geq\frac{C}{N^{c_1\alpha+\frac{1}{2}}(\alpha\log N)^{2-\frac{2}{p}}},
\end{eqnarray*}
since we are assuming $p\leq N^\frac{1}{4}$.
Taking $\alpha$ such that $c_1\alpha+\frac{1}{2}<1$ we obtain
$$
M\left(\frac{1}{s_0}\right)  \geq \frac{1}{N}
$$
if $ N\geq N_1$ and $n\geq n_0$ for some $n_0, N_1$ big enough. Therefore
  $$
      \mathbb E \max_{1 \leq i \leq N} \abs{ \skp{X_i,e_1} } \geq \tilde C \left( \log\tfrac{N}{p^{\frac{1}{4}}} \right)^{\frac{1}{p}} \geq C \left( \log N \right)^{\frac{1}{p}},
  $$
where $N\geq N_0$ and $C$ is a positive absolute constant.\\
Now we consider the case $p\geq c\frac{n}{\log N}$ or $p \geq N^{1/4}$. In that case we choose
  $$
      s_0 = \frac{1}{2|B_p^n|^{\frac{1}{n}}}.
  $$
Then 
  \begin{eqnarray*}
      M\left( \frac{1}{s_0} \right) & = & 2  \frac{|B_p^{n-1}| |B_p^n|^{\frac{1}{n}}}{|B_p^n|}  \int_{|B_p^n|^{\frac{1}{n}}}^{2|B_p^n|^{\frac{1}{n}}}  \int_{\frac{1}{t}}^{|B_p^n|^{-\frac{1}{n}}} r \left(1-|B_p^n|^{\frac{p}{n}} r^p \right)^{\frac{n-1}{p}} dr dt \\
      & \geq & 2  \frac{|B_p^{n-1}| |B_p^n|^{\frac{1}{n}}}{|B_p^n|}  \int_{\frac{7}{4}|B_p^n|^{\frac{1}{n}}}^{2|B_p^n|^{\frac{1}{n}}}  \int_{\frac{1}{t}}^{\frac{2}{3}|B_p^n|^{-\frac{1}{n}}} r \left(1-|B_p^n|^{\frac{p}{n}} r^p \right)^{\frac{n-1}{p}} dr dt \\
      & \geq & 2  \frac{|B_p^{n-1}| |B_p^n|^{\frac{1}{n}}}{|B_p^n|}  \int_{\frac{7}{4}|B_p^n|^{\frac{1}{n}}}^{2|B_p^n|^{\frac{1}{n}}} 
      \left( \frac{2}{3|B_p^n|^{\frac{1}{n}}}-\frac{1}{t}\right) \frac{1}{t}\left(1-\frac{1}{\left(\frac{3}{2}\right)^p} \right)^{\frac{n-1}{p}} dt \\
      & \geq & \frac{1}{42} \frac{|B_p^{n-1}|}{|B_p^n|} \left(1-\frac{1}{\left(\frac{3}{2}\right)^p} \right)^{\frac{n-1}{p}} 
      ~ \geq ~ C_1 n^{\frac{1}{p}} e^{\frac{n-1}{p} \log\left(1-\frac{1}{\left( \frac{3}{2} \right)^p}\right)} \\
      & \geq & C_1 n^{\frac{1}{p}} e^{-c_2\frac{n-1}{p(3/2)^p}}. 
  \end{eqnarray*}  
We want the latter expression to be greater or equal to $N^{-1}$, {\it i.e.},
  $$
     C_1 n^{\frac{1}{p}} e^{-c_2\frac{n-1}{p\left(3/2\right)^p}} \geq \frac{1}{N},
  $$
which is equivalent to
  $$
      \log N + \log(C_1) + \frac{1}{p} \log(n) \geq c_2\frac{n-1}{p\left(\frac{3}{2}\right)^p}.
  $$ 
To obtain this, it is enough to show
  $$
      \log N  \geq c_2\frac{n-1}{p\left(\frac{3}{2}\right)^p},
  $$
and since $p\geq c\frac{n}{\log N}$ and $N\leq e^{c'n}$, to obtain the latter inequality, it is enough to have
  $$
      \log N  \geq c_2\frac{n-1}{p\left(\frac{3}{2}\right)^{\frac{c}{c'}}}.
  $$     
But 
  $$
      c_2\frac{n-1}{p\left(\frac{3}{2}\right)^{\frac{c}{c'}}} \leq c_2\frac{n-1}{c\frac{n}{\log N }\left(\frac{3}{2}\right)^{\frac{c}{c'}}} \leq c_2\frac{\log N}{c\left(\frac{3}{2}\right)^{\frac{c}{c'}}} \leq \log N,
  $$
   if $c'$ is small enough. So we obtain the estimate. 
   If $p\geq N^{\frac{1}{4}}$ we immediately obtain
   $$
       C_1 n^{\frac{1}{p}} e^{-c_2\frac{n-1}{p\left(3/2\right)^p}} \geq C > \frac{1}{N},
   $$
 for $N\geq N_0$.  Therefore, in these two cases, we obtain the estimate
     $$
         \E h_{K_N}(e_j) \sim \frac{1}{|B_p^n|^{\frac{1}{n}}} \sim n^{\frac{1}{p}} \gtrsim (\log N)^{\frac{1}{p}}.
     $$
\end{proof}

\begin{remark}
In the case $p = \infty$ it is very easy to check that 
  $$
      \inf \left\{ s>0 : M\left(\frac{1}{s}\right) \leq \frac{1}{N} \right\} = \frac{ 1+\frac{1}{N} - \sqrt{\frac{2}{N}+\frac{1}{N^2}} }{2} \sim 1,
  $$
and so $\E h_{K_N}(e_j) \sim 1$.  
\end{remark}

\section{General Results}\label{GeneralResult}

Using our approach, we will now prove more general bounds for
symmetric isotropic convex bodies.  In the first theorem we assume
some mild technical conditions which are verified by the $\ell_p^n$
balls ($p\geq 2$). In this way we recover the upper estimates proved in the previous section.\\
Since $\E h_{K_N}(\theta) \sim \inf \left\{ s>0 : M_{\theta}\left( \frac{1}{s} \right) \leq \frac{1}{N} \right\}$, it seems natural to study for which value of $s$
  $$
      \int_{S^{n-1}} M_{\theta}(\tfrac{1}{s}) d\mu(\theta) = \frac{1}{N}.
  $$
As one could expect, this value
of $s$ is of the order $L_K \sqrt{\log N}$. As a consequence of
Chebychev's inequality we will obtain probability estimates for the set
of directions verifying $\E h_{K_N}(\theta) \leq C L_K \sqrt{\log N}$
or $\E h_{K_N}(\theta) \geq C L_K \sqrt{\log N}$ .   

\begin{thrm}
  Let $K$ be a symmetric and isotropic convex body, $n\leq N$, $\theta\in S^{n-1}$ and $X_1,\ldots,X_N$ be independent random vectors uniformly distributed in $K$. Define $h(t)= | K \cap \{ \skp{x,\theta}=t\}| ^{\frac{1}{n-1}}$.
  Assume that $h$ is twice differentiable and that $h'(t) \neq 0$ for all
  $t\in (0,h_K(\theta))$. Assume also that $-h'(t)/t$ is increasing, and that $h(h_K(\theta))=0$. Then,
    $$
        \mathbb E \max_{1\leq i \leq N}|\skp{X_i,\theta}| \leq C h^{-1}\left(h(0)(1-\alpha\tfrac{\log N}{n})\right),
    $$
  where $\alpha,C$, $\alpha>C$ are positive absolute constants.
\end{thrm}

\begin{proof}
  First of all notice that $h$ is a concave function. Then, using Theorem \ref{THM Schuett Werner Litvak Gordon}, we get
    $$
        M\left(\frac{1}{s}\right) = \int_{\frac{1}{h_K(\theta)}}^{\frac{1}{s}} 2 \int_{\frac{1}{t}}^{h_K(\theta)} r h(r)^{n-1} dr dt =  2 \int_{\frac{1}{h_K(\theta)}}^{\frac{1}{s}}  \int_{\frac{1}{t}}^{h_K(\theta)} \frac{r}{h'(r)} h'(r)h(r)^{n-1} drdt.
    $$
  Integration by parts yields
    $$
        M\left(\frac{1}{s}\right) = 2 \int_{\frac{1}{h_K(\theta)}}^{\frac{1}{s}} -\frac{\tfrac{1}{t}h(\tfrac{1}{t})^n}{nh'(\tfrac{1}{t})} dt - \int_{\frac{1}{h_K(\theta)}}^{\frac{1}{s}}
        \int_{\frac{1}{t}}^{h_K(\theta)} h(r)^n \frac{h'(r)-rh''(r)}{nh'(r)^2} dr dt.
    $$
  Since $h'(t)-th''(t) \geq 0$, we have
    $$
        M\left(\frac{1}{s}\right) \leq 2 \int_{\frac{1}{h_K(\theta)}}^{\frac{1}{s}} -\frac{\tfrac{1}{t}h(\tfrac{1}{t})^n}{nh'(\tfrac{1}{t})} dt = -\frac{2}{n} \int_{s}^{h_K(\theta)} \frac{h(u)^n}{uh'(u)} du
        = -\frac{2}{n} \int_{s}^{h_K(\theta)} \frac{h'(u)h(u)^n}{uh'(u)^2} du.
    $$
  Again we use integration by parts and get
    \begin{eqnarray*}
        M\left(\frac{1}{s}\right) & \leq & \frac{2h(s)^{n+1}}{n(n+1)sh'(s)^2} - \frac{2}{n(n+1)} \int_{s}^{h_K(\theta)} \frac{h(u)^{n+1}(h'(u)+2uh''(u))}{u^2h'(u)^3} du \\
        & \leq	 & \frac{2sh(s)^{n+1}}{n(n+1)s^2h'(s)^2}.
    \end{eqnarray*}
  Furthermore, since we have $h'(t)-th''(t) \geq 0$, we get
    $$
        -sh'(s) = |sh'(s)| = \int_0^s - h'(t)-th''(t) dt \geq -2\int_0^sh'(t)dt = 2(h(0)-h(s)).
    $$
  Thus
    $$
        M\left(\frac{1}{s}\right) \leq \frac{sh(s)^{n-1}}{2n(n+1)(\frac{h(0)}{h(s)}-1)^2} = \frac{se^{(n-1) \log\frac{h(s)}{h(0)} } |K\cap\theta^{\bot}|h(s)^2 }{2n(n+1)(1-\frac{h(s)}{h(0)})^2h(0)^2}.
    $$
  Choosing
    $$
        s_0 = h^{-1}\left( h(0)(1-\alpha\tfrac{\log N}{n}) \right),
    $$
  there exists a positive constant $c_1$ such that
    $$
        M\left(\frac{1}{s_0}\right) \leq C \frac{s_0|K\cap\theta^{\bot}|}{N^{c_1\alpha}\alpha^2(\log N)^2}.
    $$
  Since $K$ is isotropic, $s_0 \leq (n+1)L_K$. Therefore,
    $$
        M\left(\frac{1}{s_0}\right) \leq C \frac{nL_K |K\cap\theta^{\bot}| }{N^{c_1\alpha}\alpha^2(\log N)^2}.
    $$
  By Hensley's result (see \cite{key-H}), $L_K \sim \frac{1}{|K\cap\theta^{\bot}|}$, and because $n\leq N$, we have
    $$
        M\left(\frac{1}{s_0}\right) \leq \frac{ CN }{ N^{c_1\alpha} \alpha^2(\log N)^2} = \frac{ C }{ N^{c_1\alpha-1} \alpha^2(\log N)^2}.
    $$
  Taking $\alpha$ so that $c_1\alpha>2$, we have $M(\tfrac{1}{s_0}) \leq \tfrac{1}{N}$ for $N\geq N_0$ for some $N_0\in \N$ big enough.
\end{proof}

With the method, introduced in Section \ref{Preliminaries}, we are
also able to prove the following general result, which will lead us to
estimates of the support function for some directions of random polytopes in symmetric isotropic convex bodies: 

\begin{thrm} \label{THM general result 2}
  Let $n \leq N \leq e^{\sqrt{n}}$, $K$ be a symmetric isotropic convex body in $\R^n$ and let $X_1,\ldots, X_N$ be independent random variables
  uniformly distributed in $K$. Then,
    $$
       \int_{S^{n-1}} M_{\theta}\left(\frac{1}{C_1L_K \sqrt{\log N}}\right) d\mu(\theta) \leq \frac{1}{N},
    $$
  and
    $$
       \int_{S^{n-1}} M_{\theta}\left(\frac{1}{C_2 L_K \sqrt{\log N}}\right) d\mu(\theta) \geq \frac{1}{N},
    $$
  where $C_1, C_2$ are positive absolute constants.\\
Consequently, if $\tilde{s}$ is chosen such that
     $$
         \int_{S^{n-1}} M_{\theta}(\tfrac{1}{\tilde{s}}) d\mu(\theta) = \frac{1}{N},
     $$
   then $\tilde{s} \sim L_K \sqrt{\log N}$.      
\end{thrm}

 In order to prove this theorem we need the following proposition:

\begin{prop}
 Let $K$ be a symmetric convex body in $\R^n$ of volume 1. Let $s>0$, $\theta\in S^{n-1}$ and $M_{\theta}$ be the Orlicz function associated to the random variable
 $\skp{X,\theta}$, where $X$ is uniformly distributed in $K$. Then,
   \begin{equation}\label{EQU equivalent representation of integral over M}
       \int_{S^{n-1}} M_{\theta}\left(\frac{1}{s}\right) d\mu(\theta) = \int_{K} M_{\skp{\theta,e_1}}\left(\frac{\norm{x}_2}{s}\right) dx, 
   \end{equation}
 where $M_{\skp{\theta,e_1}}$ is the Orlicz function associated to the random variable $\skp{\theta,e_1}$ with $\theta$ uniformly distributed on $S^{n-1}$.  For any 
 $s\leq \norm{x}_2$
   \begin{equation}\label{EQU formula for M in the general case}
       M_{\skp{\theta,e_1}}\left(\frac{\norm{x}_2}{s}\right) = \frac{2w_{n-1}}{nw_n} \int_0^{\cos^{-1}(\frac{s}{\norm{x}_2})} \frac{\sin^n y}{\cos^2 y} dy,
   \end{equation}
 and $0$ otherwise.      
\end{prop} 
\begin{proof}
 Using the definition of $M_{\theta}$, we obtain
    \begin{eqnarray*}
        &&\int_{S^{n-1}} M_{\theta}\left(\frac{1}{s}\right) d\mu(\theta) \cr 
        & = & \int_{S^{n-1}} \int_0^{\frac{1}{s}} \int_K  \mathbbm 1_{\{\abs{\skp{x,\theta}}\geq \frac{1}{t}\}}(x,\theta,t) 
        \abs{\skp{x,\theta}} dx dt d\mu(\theta) \cr
        & = &  \int_K \int_0^{\frac{1}{s}}  \int_{S^{n-1}}  \mathbbm 1_{\{\abs{\skp{x,\theta}}\geq \frac{1}{t}\}}(x,\theta,t) 
        \abs{\skp{x,\theta}}  d\mu(\theta) dt dx \cr
        & = &  \int_K \int_0^{\frac{\norm{x}_2}{s}}  \int_{S^{n-1}}  \mathbbm 1_{\left\{\abs{\skp{\frac{x}{\norm{x}_2},\theta}} \geq \frac{1}{u}\right\}}(x,\theta,u) 
        \abs{\skp{\frac{x}{\norm{x}_2},\theta}}  d\mu(\theta) du dx,
    \end{eqnarray*}
  where the last equality is obtained by the change of variable $t=\frac{u}{\norm{x}_2}$. Hence, by the rotationally invariance of $S^{n-1}$,
    \begin{equation*} 
        \int_{S^{n-1}} M_{\theta}\left(\frac{1}{s}\right) d\mu(\theta) = \int_K M_{ \skp{\theta,e_1}  }\left(\tfrac{\norm{x}_2}{s}\right) dx.
    \end{equation*}
  Now, let us compute $M_{ \skp{\theta,e_1}  }$. For any $s>1$, otherwise the function is $0$, we have
    \begin{eqnarray*}
      M_{\skp{\theta,e_1}}(s) & = & \int_1^{s} \int_{S^{n-1}\cap \{ \skp{e_1,\theta} \geq \frac{1}{t} \}} d\mu(\theta) dt \\
      & = & 2 \int_1^s \frac{(n-1)w_{n-1}}{nw_{n}} \int_{\frac{1}{t}}^1 r(1-r^2)^{\frac{n-3}{2}} dr dt \\
      & = & \frac{2w_{n-1}}{nw_n} \int_1^s (1-\tfrac{1}{t^2})^{\frac{n-1}{2}} dt. \\
    \end{eqnarray*}
  The change of variables $\frac{1}{t} = \cos y$ yields
    \begin{equation*} 
       M_{\skp{\theta,e_1}}(s) =  \frac{2w_{n-1}}{nw_n} \int_0^{\cos^{-1}(\tfrac{1}{s})} \frac{\sin^n y}{\cos^2 y} dy.
    \end{equation*}    
\end{proof}

Given that the expected mean width of $K_N$ is minimized when $K=D_2^n$, it is natural to expect that given $s$, the average $\int_{S^{n-1}}M_\theta\left(\frac{1}{s}\right)d\mu(\theta)$ would also be minimized when $K=D_2^n$. We prove it, using this representation, in the following:

\begin{cor}
Let $K$ be a symmetric convex body in $\R^n$ of volume 1 and let $s>0$. Then
$$
\int_{S^{n-1}}M_\theta\left(\frac{1}{s}\right)d\mu(\theta)\geq\int_{S^{n-1}}M_{D_2^n,\theta}\left(\frac{1}{s}\right)d\mu(\theta)=M_{D_2^n,e_1}\left(\frac{1}{s}\right),
$$
where $M_{D_2^n,\theta}$ denotes the Orlicz function associated to $D_2^n$.
\end{cor}

\begin{proof}
By (\ref{EQU equivalent representation of integral over M}) and the facts that $M_{\skp{\theta,e_1}}$ is increasing  and $|K|=|D_2^n|=1$ we have that if $r_n$ is the radius of $D_2^n$
\begin{eqnarray*}
\int_{S^{n-1}} M_{\theta}\left(\frac{1}{s}\right) d\mu(\theta) &= &\int_{K} M_{\skp{\theta,e_1}}\left(\frac{\norm{x}_2}{s}\right) dx\cr
&=&\int_{K\cap D_2^n} M_{\skp{\theta,e_1}}\left(\frac{\norm{x}_2}{s}\right) dx+\int_{K\backslash D_2^n} M_{\skp{\theta,e_1}}\left(\frac{\norm{x}_2}{s}\right) dx\cr
&\geq&\int_{K\cap D_2^n} M_{\skp{\theta,e_1}}\left(\frac{\norm{x}_2}{s}\right) dx+|K\backslash D_2^n|M_{\skp{\theta,e_1}}\left(\frac{r_n}{s}\right) \cr
&=&\int_{K\cap D_2^n} M_{\skp{\theta,e_1}}\left(\frac{\norm{x}_2}{s}\right) dx+|D_2^n\backslash K|M_{\skp{\theta,e_1}}\left(\frac{r_n}{s}\right) \cr
&\geq&\int_{K\cap D_2^n} M_{\skp{\theta,e_1}}\left(\frac{\norm{x}_2}{s}\right) dx+\int_{D_2^n\backslash K} M_{\skp{\theta,e_1}}\left(\frac{\norm{x}_2}{s}\right) dx\cr
&=&\int_{D_2^n} M_{\skp{\theta,e_1}}\left(\frac{\norm{x}_2}{s}\right) dx=\int_{S^{n-1}} M_{D_2^n,\theta}\left(\frac{1}{s}\right) d\mu(\theta) .
\end{eqnarray*}
\end{proof}
Now, we give the proof of Theorem \ref{THM general result 2}:

\begin{proof}
  By (\ref{EQU formula for M in the general case}), if $\norm{x}_2 \geq s$, we have
    \begin{eqnarray*}
      M_{\skp{\theta,e_1}}\left(\frac{\norm{x}_2}{s}\right) 
      & = &  \frac{2w_{n-1}}{nw_n} \int_0^{\cos^{-1}(\tfrac{s}{\norm{x}_2})} \frac{\sin^n y}{\cos^2 y} dy.
    \end{eqnarray*}
  Integration by parts yields
    \begin{eqnarray*}
      M_{\skp{\theta,e_1}}\left(\frac{\norm{x}_2}{s}\right) 
      & = &  \frac{2w_{n-1}}{nw_n} \left[ \frac{(\sin y)^{n-1}}{\cos y} \Big|_{0}^{\cos^{-1}(\tfrac{s}{\norm{x}_2})} 
      - (n-1) \int_0^{\cos^{-1}(\tfrac{s}{\norm{x}_2})} (\sin y)^{n-2} dy \right] \\
      & = & \frac{2w_{n-1}}{nw_n} \left[ \frac{\norm{x}_2}{s} \left(1-\frac{s^2}{\norm{x}_2^2}\right)^{\frac{n-1}{2}}  
      - (n-1) \int_0^{\cos^{-1}(\tfrac{s}{\norm{x}_2})} (\sin y)^{n-2} dy \right] .
    \end{eqnarray*}
  We start with the upper bound where we will use Paouris' result about the concentration of mass on isotropic convex bodies from \cite{key-P}. First of all, we have
    $$ 
        M_{\skp{\theta,e_1}}\left(\frac{\norm{x}_2}{s}\right) \leq \frac{2w_{n-1}}{nw_n} 
        \frac{\norm{x}_2}{s} \left(1-\frac{s^2}{\norm{x}_2^2}\right)^{\frac{n-1}{2}}.
    $$ 
  From (\ref{EQU equivalent representation of integral over M}) and since $M_{\skp{\theta,e_1}}\left(\frac{\norm{x}_2}{s}\right) = 0$ for $s>\norm{x}_2$, we get
    \begin{eqnarray*}
        \int_{S^{n-1}} M_{\theta}(\tfrac{1}{s}) d\mu(\theta) & \leq & 
        \int_{K\setminus sB_2^n} \frac{2w_{n-1}}{nw_n} \frac{\norm{x}_2}{s} \left(1-\frac{s^2}{\norm{x}_2^2}\right)^{\frac{n-1}{2}} dx \\
        & \leq &\int_{K\setminus sB_2^n} \frac{2w_{n-1}}{nw_n} \frac{\norm{x}_2}{s} e^{-\frac{n-1}{2}\frac{s^2}{\norm{x}_2^2}} dx \\
        & \leq &\int_{K} \frac{2w_{n-1}}{nw_n} \frac{\norm{x}_2}{s} e^{-\frac{n-1}{2}\frac{s^2}{\norm{x}_2^2}} dx.
    \end{eqnarray*} 
  We choose $s_0 = \sqrt{\alpha} L_K \sqrt{\log N}$, with $\alpha>0$ a constant to be chosen later. Then, if $N \leq e^{\sqrt{n}}$,
    \begin{eqnarray*}
        \int_{S^{n-1}} M_{\theta}(\tfrac{1}{s_0}) d\mu(\theta) & \leq & \frac{2w_{n-1}}{nw_n} \frac{1}{\sqrt{\alpha} L_K \sqrt{\log N}}
       \int_K \norm{x}_2 e^{-\frac{c_1\alpha n L_K^2}{\norm{x}_2^2} \log N} dx \\
       & = & \frac{2w_{n-1}}{nw_n} \frac{1}{\sqrt{\alpha} L_K \sqrt{\log N}} 
       \left[ \int_{K\cap \gamma\sqrt{n}L_KB_2^n} \norm{x}_2 e^{-\frac{c_1\alpha n L_K^2}{\norm{x}_2^2} \log N} dx + \right. \\
       && \left. + \int_{K\setminus \gamma\sqrt{n}L_KB_2^n} \norm{x}_2 e^{-\frac{c_1\alpha n L_K^2}{\norm{x}_2^2} \log N} dx  \right] \\
       & \leq & \frac{2w_{n-1}}{nw_n} \frac{1}{\sqrt{\alpha} L_K \sqrt{\log N}} 
       \left[ \frac{\gamma \sqrt{n}L_K}{N^{\frac{c_1\alpha}{\gamma^2}}} + n L_K e^{-c_1\sqrt{n}\gamma}\right] \\
       & \leq & \frac{C}{\sqrt{\alpha}\sqrt{\log N}} \left[ \frac{\gamma}{N^{\frac{c_1\alpha}{\gamma^2}}} + \frac{\sqrt{n}}{N^{c_1\gamma}} \right] \\
       & \leq & \frac{C}{\sqrt{\alpha}\sqrt{\log N}} \left[ \frac{\gamma}{N^{\frac{c_1\alpha}{\gamma^2}}} + \frac{1}{N^{c_1\gamma-\frac{1}{2}}} \right]. \\
    \end{eqnarray*} 
  We choose $\gamma>0$ such that $c_1\gamma - \frac{1}{2}>1$ and then $\alpha>0$ so that $\frac{c_1\alpha}{\gamma^2}>1$. Then, 
    $$
         \int_{S^{n-1}} M_{\theta}(\tfrac{1}{s_0}) d\mu(\theta)  \leq  \frac{C}{\sqrt{\alpha}\sqrt{\log N}} \left[ \frac{\gamma}{N^{\frac{c_1\alpha}{\gamma^2}}} + \frac{1}{N^{c_1\gamma-\frac{1}{2}}} \right] \\
         \leq \frac{1}{N},
    $$
  for $N\leq e^{\sqrt{n}}$ and $N \geq N_0$.   \\
  To prove the lower bound we use the recursion formula  (\ref{COR General Recursion for Sin k steps}).  For $\norm{x}_2 \geq s $, 
  and any $k\in \N$,
    \begin{eqnarray*}
      && M_{\skp{\theta,e_1}}\left(\frac{\norm{x}_2}{s}\right) \\
      & \geq & \frac{2w_{n-1}}{nw_n} \frac{\norm{x}_2}{s(n+1)} \left[ \left(1-\frac{s^2}{\norm{x}_2^2}\right)^{\frac{n+1}{2}} 
      + \left(1-\frac{3}{n+3} \right)\left(1-\frac{s^2}{\norm{x}_2^2}\right)^{\frac{n+3}{2}}+ \right. \\
      && \left. + \cdots + \left(1-\frac{3}{n+3} \right)\cdots\left(1-\frac{3}{n+2k+1} \right)\left(1-\frac{s^2}{\norm{x}_2^2}\right)^{\frac{n+2k+1}{2}} \right] \\
      & \geq & \frac{2w_{n-1}}{nw_n} \frac{\norm{x}_2(k+1)}{s(n+1)}\left(1-\frac{3}{n+2k+1} \right)^{k}\left(1-\frac{s^2}{\norm{x}_2^2}\right)^{\frac{n+2k+1}{2}}.
    \end{eqnarray*} 
  Taking $k=n$
    \begin{eqnarray*}
        M_{\skp{\theta,e_1}}\left(\frac{\norm{x}_2}{s}\right) & \geq &  \frac{2w_{n-1}}{nw_n} \frac{\norm{x}_2}{s} \left(1-\frac{3}{n+2k+1} \right)^{n}
        \left(1-\frac{s^2}{\norm{x}_2^2}\right)^{\frac{3n+1}{2}} \\
        & \geq & \frac{Cw_{n-1}}{nw_n} \frac{\norm{x}_2}{s}\left(1-\frac{s^2}{\norm{x}_2^2}\right)^{\frac{3n+1}{2}}.
    \end{eqnarray*} 
  Thus
    \begin{eqnarray*}
      \int_{S^{n-1}} M_{\theta}(\tfrac{1}{s}) d\mu(\theta) 
      & \geq & \int_{K\setminus sB_2^n} \frac{Cw_{n-1}}{nw_n} \frac{\norm{x}_2}{s}\left(1-\frac{s^2}{\norm{x}_2^2}\right)^{\frac{3n+1}{2}} dx \\
      & \geq &\int_{K\setminus 2sB_2^n} \frac{Cw_{n-1}}{nw_n}  \frac{\norm{x}_2}{s}\left(1-\frac{s^2}{\norm{x}_2^2}\right)^{\frac{3n+1}{2}} dx \\
      & \geq & \int_{K\setminus 2sB_2^n} \frac{Cw_{n-1}}{nw_n}  \frac{\norm{x}_2}{s}e^{-c_4n\frac{s^2}{\norm{x}_2^2}} dx.  
    \end{eqnarray*} 
  Take $s_1 = \sqrt{\beta} L_K \sqrt{\log N}$, $\beta>0$ a constant to be chosen later. Then
    $$
        \int_{S^{n-1}} M_{\theta}(\tfrac{1}{s_1}) d\mu(\theta) \geq 
        \int_{K\setminus 2\sqrt{\beta} L_K \sqrt{\log N}B_2^n} \frac{Cw_{n-1}}{nw_n}  \frac{\norm{x}_2}{\sqrt{\beta} L_K \sqrt{\log N}}
        e^{-c_4n\frac{\beta L_K^2\log N}{\norm{x}_2^2}} dx.
    $$ 
  Using the small ball probability result proved in \cite{key-P2} there exists a constant $c_5>0$ such that 
    $$
        |K\setminus c_5\sqrt{n}L_K B_2^n| \geq \frac{1}{2},
    $$
  for $N\leq e^{n}$. Therefore,
    \begin{eqnarray*}
       && \int_{K\setminus 2\sqrt{\beta} L_K \sqrt{\log N}B_2^n} \frac{Cw_{n-1}}{nw_n}  \frac{\norm{x}_2}{\sqrt{\beta} L_K \sqrt{\log N}}
        e^{-c_4n\frac{\beta L_K^2\log N}{\norm{x}_2^2}} dx \\
        & \geq & \int_{K\setminus c_5\sqrt{n} L_K B_2^n} \frac{Cw_{n-1}}{nw_n}  \frac{\norm{x}_2}{\sqrt{\beta} L_K \sqrt{\log N}}
        e^{-c_4n\frac{\beta L_K^2\log N}{\norm{x}_2^2}} dx \\
        & \geq & \frac{C'}{\sqrt{\beta}\sqrt{\log N}} e^{-c_6\beta \log N} |K\setminus c_5\sqrt{n} L_K B_2^n| \\
        & \geq & \frac{C''}{N^{c_6\beta}\sqrt{\beta}\sqrt{\log N}},
    \end{eqnarray*}
  where the inequality before the last one holds because $\norm{x}_2^2 \geq c_5^2 n L_K^2$. We take $\beta$ small enough, so that $c_6\beta<1$ and 
  $2\sqrt{\beta}\sqrt{\log N} \leq c_5 \sqrt{n}$. Then
    $$
        \frac{C''}{N^{c_6\beta}\sqrt{\beta}\sqrt{\log N}} \geq \frac{1}{N},
    $$
  for $N\geq N_0$ and $N \leq e^{n}$. Hence,
    $$
        \int_{S^{n-1}} M_{\theta}(\tfrac{1}{s_1}) d\mu(\theta) \geq \frac{1}{N}.
    $$              
\end{proof}

Obviously, the theorem implies that there are directions $\theta_1,
\theta_2 \in S^{n-1}$ such that the expectation of the support
function in those directions is bounded from above and below
respectively by a constant times $L_K\sqrt{\log N}$. In
Corollary \ref{COR measure estimate} we give estimates for the measure of the set of directions
verifying such estimates. However, we don't think that the estimate we
give for the measure of the set of directions verifying the lower bound
is optimal.

\begin{proof}[of Corollary \ref{COR measure estimate}]
  To prove that the upper bound is true for most directions we
  proceed as in the proof of Theorem \ref{THM general result 2}. We choose $s_0$ like there, and $\alpha$, $\gamma$ so that $c_1\gamma - \frac{1}{2}> 2(r+1)$ and 
  $\frac{c_1\alpha}{\gamma^2}>2(r+1)$ and obtain
    $$
        \int_{S^{n-1}} M_{\theta}(\tfrac{1}{s_0}) d\mu(\theta) \leq \frac{1}{N^{r+1}}.
    $$
  Then, by Chebychev's inequality,
    $$
        \frac{1}{N^{r+1}} \geq  \int_{S^{n-1}} M_{\theta}(\tfrac{1}{s_0}) d\mu(\theta) \geq \frac{1}{N} \mu \left\{ \theta \in S^{n-1} : M_{\theta}\left( \frac{1}{s_0} \right) > \frac{1}{N} \right\}.
    $$
  Thus
     $$
         \mu \left\{ \theta \in S^{n-1} : M_{\theta}\left( \frac{1}{s_0} \right) \leq \frac{1}{N} \right\} \geq 1-\frac{1}{N^{r}}
     $$    
and so
$$
   \mu \left\{ \theta \in S^{n-1} :\E h_{K_N}(\theta) \leq
     C_1(r)L_K\sqrt{\log N} \right\} \geq 1-\frac{1}{N^{r}}.
$$
To prove the probability estimate for the lower bound we can assume
that $r<1$. We proceed as
in  Theorem \ref{THM general result 2}. We choose $s_1$ like there and take $\beta$ small enough
so that $c_6\beta<r$. We obtain
  $$
      \int_{S^{n-1}} M_{\theta}\left( \frac{1}{s_1} \right) d\mu(\theta) > \frac{1}{N^{r}}.
  $$
 Then, for any  decreasing, positive and concave function $f$ we get
  $$
       f\left( \int_{S^{n-1}} M_{\theta}\left( \frac{1}{s_1} \right) d\mu(\theta) \right) < f\left( \frac{1}{N^{r}} \right).
  $$
Using Jensen's inequality this yields
  \begin{eqnarray*}
    f\left( \frac{1}{N^{r}} \right) & \geq &  \int_{S^{n-1}} f\left( M_{\theta}\left( \frac{1}{s_1} \right) \right) d\mu(\theta) \\
    & \geq & f\left( \frac{1}{N} \right) \mu\left\{ \theta\in S^{n-1} : f\left( M_{\theta}\left( \frac{1}{s_1} \right) \right) > f\left( \frac{1}{N}\right) \right\} \\
    & = & f\left( \frac{1}{N} \right) \mu\left\{ \theta\in S^{n-1} :  M_{\theta}\left( \frac{1}{s_1} \right)  <  \frac{1}{N} \right\} .
  \end{eqnarray*} 
Thus
  $$
      \mu\left\{ \theta\in S^{n-1} :  M_{\theta}\left( \frac{1}{s_1} \right)  <  \frac{1}{N} \right\} \leq \frac{f\left( \frac{1}{N^{r}} \right)}{f\left( \frac{1}{N} \right)},
  $$
and therefore
  $$
      \mu\left\{ \theta\in S^{n-1} :  M_{\theta}\left( \frac{1}{s_1} \right)  \geq   \frac{1}{N} \right\} \geq 1- \frac{f\left( \frac{1}{N^{r}} \right)}{f\left( \frac{1}{N} \right)}.
  $$
This means that
  $$
       \mu\left\{ \theta\in S^{n-1} :  \E h_{K_N}(\theta) \geq cs_1 \right\} \geq 1- \frac{f\left( \frac{1}{N^{r}} \right)}{f\left( \frac{1}{N} \right)}.
  $$
We choose $f(t) = -at + a \max_{\theta\in S^{n-1}} M_{\theta}(\tfrac{1}{s_1})$, $a>0$. Then
  $$
      \frac{f\left( \frac{1}{N^{r}} \right)}{f\left( \frac{1}{N} \right)} = \frac{-\frac{1}{N^{r}}+\max_{\theta\in S^{n-1}} M_{\theta}(\tfrac{1}{s_1})}{-\frac{1}{N}+\max_{\theta\in S^{n-1}} M_{\theta}(\tfrac{1}{s_1})},
  $$
and thus
  $$
      1- \frac{f\left( \frac{1}{N^{r}} \right)}{f\left( \frac{1}{N} \right)} = \frac{\frac{1}{N^{r}} - \frac{1}{N} }{\max_{\theta\in S^{n-1}} M_{\theta}(\tfrac{1}{s_1})-\frac{1}{N}}.
  $$
From H\"older's inequality we obtain
  \begin{eqnarray*}
      M_{\theta}\left( \frac{1}{s_1}\right) & = & \int_0^{\frac{1}{s_1}} \int_{K\cap \{ \abs{ \skp{x,\theta}} \geq \frac{1}{t} \} } \abs{\skp{x,\theta} } dx dt \\
      & \leq &  \int_0^{\frac{1}{s_1}} \int_{K} \abs{\skp{x,\theta} } dx dt \\
      & \leq & \int_0^{\frac{1}{s_1}} L_K dt ~=~ \frac{L_K}{s_1}.
  \end{eqnarray*} 
Because of our choice of $s_1$ we get
  $$
      M_{\theta}\left( \frac{1}{s_1}\right) \leq \frac{C(r)}{\sqrt{\log N}}.
  $$   
Therefore
  $$
      1- \frac{f\left( \frac{1}{N^{r}} \right)}{f\left( \frac{1}{N} \right)} \geq \frac{\frac{1}{N^{r}} - \frac{1}{N} }{\frac{C(r)}{\sqrt{\log N}}-\frac{1}{N}} \geq \frac{C'(r)\sqrt{\log N}}{N^{r}}.
  $$
This yields
  $$
       \mu\left\{ \theta\in S^{n-1} :  \E h_{K_N}(\theta) \geq  C_2(r)L_K \sqrt{\log N} \right\} \geq \frac{C(r)\sqrt{\log N}}{N^{r}}.
  $$
\end{proof}

\proof[Acknowledgements]
This work was done while the authors were postdoctoral fellows at the Department of Mathematical and Statistical Sciences at University of Alberta. We would like to thank the department for providing such good environment and working conditions. Especially, we would like to thank Nicole Tomczak-Jaegermann and Alexander Litvak for pointing out the problem to us and for useful comments.

\end{document}